\documentclass[12pt]{article} 

\usepackage[margin=2.5cm]{geometry}
\usepackage[english]{babel}
\usepackage{hyperref}
\usepackage{amsmath}
\usepackage{amsthm}
\usepackage{amssymb}
\usepackage{ucs}
\usepackage{enumerate}
\usepackage[utf8]{inputenc}
\usepackage[T1]{fontenc}

\DeclareMathOperator{\rng}{\mathrm{rng}}

\newcommand{\MAlg}{\mathrm{MAlg}}
\newcommand{\Aut}{\mathrm{Aut}}

  \newcommand{\R}{\mathbb R}
  \newcommand{\N}{\mathbb N}
  
  \newcommand{\Z}{\mathbb Z}
  
  \newcommand{\LL}{\mathrm L}

  \newcommand{\id}{\mathrm{id}}

 \newcommand{\dom}{\mathrm{dom}\;}

  \newcommand{\inv}{^{-1}}

  \renewcommand{\leq}{\leqslant}
  \renewcommand{\geq}{\geqslant}
  \newcommand{\abs}[1]{\left\lvert #1\right\rvert}

  \newcommand{\act}{\curvearrowright}

  \newcommand{\la}{\left\langle}
  \newcommand{\ra}{\right\rangle}
  \newcommand{\into}{\hookrightarrow}

  \DeclareMathOperator{\Sub}{\mathrm{Sub}}
  \DeclareMathOperator{\Subfin}{\mathrm{Sub}_{\mathrm{fin}}}
    \DeclareMathOperator{\Subhyp}{\mathrm{Sub}_{\mathrm{hyp}}}

\newtheorem{theoremi}{Theorem}				

\newtheorem{theorem}{Theorem}
\numberwithin{theorem}{section}

\newtheorem{corollary}[theorem]{Corollary}
\newtheorem{lemma}[theorem]{Lemma}

\newtheorem{proposition}[theorem]{Proposition}
\newtheorem*{claim}{Claim}

\theoremstyle{definition}

\newtheorem{question}[theorem]{Question}

\newtheorem*{ack}{Acknowledgements}

\newenvironment{cproof}{\begin{proof}[Proof of the
		claim]}{\end{proof}}

\newtheorem{definition}[theorem]{Definition}
\newtheorem*{remark}{Remark}
\newtheorem{example}[theorem]{Example}

\title{On dense orbits in the space of subequivalence relations}

\author{François Le Maître}
\date{}
\begin{document}

\maketitle

\begin{abstract}
	We first explain how to endow the space of subequivalence relations of 
	any non-singular countable equivalence relation with a Polish topology, extending
	the framework of Kechris' recent monograph on subequivalence relations of probability
	measure-preserving (p.m.p.) countable equivalence relations.
	We then restrict to p.m.p.\ equivalence relations  and discuss dense orbits therein 
	for the natural action of the full group and of the automorphism group of the relation. 
	Our main result is a characterization of the subequivalence relations having a dense orbit in the space of subequivalence relations of the ergodic hyperfinite p.m.p.\ equivalence relation. We also show that in this setup, all orbits under the full group action are meager.
	We finally provide a few Borel complexity calculations of natural subsets in spaces
	of subequivalence relations using a natural metric we call the uniform metric. This answers some questions from an earlier version of Kechris' monograph.
\end{abstract}

\section{Introduction}

Measured group theory can roughly be defined as the study of countable groups through the partitions into orbits 
associated to their free probability measure preserving (p.m.p.) actions on standard probability space.
 A fundamental notion there is that of that of an \emph{orbit subgroup}: a countable group $\Lambda$ is an orbit subgroup of another countable group $\Gamma$ when $\Lambda$ admits a free p.m.p.\ action all of whose orbits are contained in those of a free p.m.p.\ action of $\Gamma$. 
 In this more flexible framework, while the Ornstein-Weiss theorem tells us that all countable infinite amenable subgroups are orbit subgroups of one another \cite{ornsteinErgodicTheoryAmenable1980}, the Gaboriau-Lyons theorem characterizes non amenable groups as being exactly those which admit the free group on two generators as an orbit subgroup \cite{gaboriauMeasurablegrouptheoreticSolutionNeumanns2009}.
Their result has had many applications, allowing to extend results which were only known for groups containing free subgroups to the optimal class of all non-amenable groups, see e.g.\ \cite{ioanaSubequivalenceRelationsPositivedefinite2009,sewardEveryActionNonamenable2014}.

This motivates the following more general question: given a partition of the space into orbits, what are the possible subpartitions into orbits? 
In more precise terms, let us call an equivalence relation on a standard probability space $(X,\mu)$ non-singular when it comes from an action of a countable group which preserves $\mu$ null-sets, and p.m.p.\ when it comes from an action of a countable group which actually preserves the measure $\mu$. 
We can now ask our question more precisely: given a non-singular equivalence relation, what are its possible subequivalence relations? 
An essential step towards answering such a question is Kechris' remarkable result that when the ambient equivalence relation is p.m.p., the space of subequivalence relations can be endowed with a natural Polish topology. 
This allows descriptive set theory to enter the picture, and as such it is the first building block in Kechris' monograph on 
spaces of subequivalence relations, whose initial version appeared online in 2013 \cite{kechrisSpaceMeasurepreservingEquivalence}.
We begin by extending this fundamental result to the non-singular case.

\begin{theoremi}[see Thm.~\ref{thm: sub R is Polish}]\label{thmi: Polish space of subeq} 
	Let $\mathcal R$ be a non-singular equivalence relation on a standard probability space. Then its space $\Sub(\mathcal R)$ of subequivalence relations is a Polish space for the topology induced by the measure algebra of $\mathcal R$. 
\end{theoremi}

We endow $\mathcal R$ with the usual measure obtained by integrating cardinality of vertical fibers, allowing us to make sense of its measure algebra in the above result. 
In particular, we are identifying subequivalence relations which coincide on a conull set.
We refer the reader to Section \ref{sec: preliminaries} for details. Our description of the topology on the space of subequivalence relations was already known to Kechris in the p.m.p.\ case, see \cite[Sec.\ 4.4(1)]{kechrisSpaceMeasurepreservingEquivalence}. 

We actually provide two proofs of Theorem \ref{thmi: Polish space of subeq}. The first one is very short and uses the fact that in a measure algebra, converging sequences always admit subsequences whose limit is actually equal to their $\liminf$ (see Proposition \ref{prop: convergence}). 
This observation  provides a unifying viewpoint to some special cases noted by Kechris, see \cite[Thm.\ 5.1 and Thm.\ 18.3]{kechrisSpaceMeasurepreservingEquivalence}. 
The second proof is somehow more natural in that we directly show that the various axioms of being a subequivalence relation define closed subsets in the measure algebra. 
Along the way, we study a natural non-Hausdorff topology on the measure algebra that we call the lower topology. This topology has some nice continuity properties which also allow us to give a slick proof of the fact that the map which associates to a subset of an equivalence relation the equivalence relation it generates is Baire class one (see Proposition \ref{prop: continuity of generating}, 
which is a natural generalization of \cite[Prop.~19.1]{kechrisSpaceMeasurepreservingEquivalence}). 

Our next result is in the p.m.p.\ setup, and was motivated by the question of dense orbits in the space of subequivalence relations, asked by Kechris in his monograph. Indeed, it is very natural to try to understand subequivalence relations up to isomorphism, and this is precisely what the action of the automorphism group of $\mathcal R$ on subequivalence relations of $\mathcal R$ encodes. As it turns out, when $\mathcal R$ is the ergodic hyperfinite p.m.p.\ equivalence relation we obtain the following complete characterization of dense orbits for the natural actions of both the full group $[\mathcal R]$ and the automorphism group $\Aut(\mathcal R)$.

\begin{theoremi}[see Cor.\ \ref{cor: dense orbit in R0}]\label{thmi: dense orbit}
	For a subequivalence relation $\mathcal S$ of the hyperfinite ergodic p.m.p.\ equivalence relation $\mathcal R_0$, the following are equivalent:
\begin{enumerate}[(i)]
	\item $\mathcal S$ is aperiodic and has everywhere infinite index in $\mathcal R_0$;
	\item The $[\mathcal R_0]$-orbit of $\mathcal S$ is dense in $\Sub(\mathcal R_0)$;
	\item The $\Aut(\mathcal R_0)$-orbit of $\mathcal S$ is dense in $\Sub(\mathcal R_0)$.
\end{enumerate}
\end{theoremi}

In this theorem, $\mathcal S$ being everywhere of infinite index in $\mathcal R$ means that for every $A\subseteq X$ of positive measure, the restriction of $\mathcal S$ to $A$ has infinite index in the restriction of $\mathcal R$ to $A$ (almost every $\mathcal R_{\restriction A}$-class splits into infinitely many $\mathcal S_{\restriction A}$-classes; see the end of Section \ref{sec: cond meas and index} for more on this).

\begin{example}\label{ex: diffuse example}
	View $\mathcal R_0$ as the cofinite equivalence relation on $X=\{0,1\}^\N$ endowed with the probability measure $\mu=(\frac 12(\delta_0+\delta_1))^{\otimes\N}$, namely $((x_n),(y_n)\in\mathcal R_0$ iff there is $N\in\N$ such that $x_n=y_n$ for all $n\geq N$. 
	Define $\mathcal S$ by $((x_n),(y_n))\in\mathcal S$ iff there is $N\in\N$ such that $x_{2n}=y_{2n}$ for all $n\geq N$. Then $\mathcal S$ is both aperiodic and has everywhere infinite index in $\mathcal R_0$ (see Proposition \ref{prop: everywhere infinite index from product}), so by the above theorem it has a dense orbit in $\Sub(\mathcal R_0)$.
\end{example}

\begin{example}\label{ex: ergodic example}
	View $\mathcal R_0$ as the  equivalence relation on $X=\{0,1\}^{\Z^2}$ endowed with the probability measure $\mu=(\frac 12(\delta_0+\delta_1))^{\otimes\Z^2}$ generated by the Bernoulli shift of $\Z^2$.
	Then since the shift is mixing, the subequivalence relation $\mathcal S$ generated by $\Z\leq \Z^2$ is ergodic. By construction $\mathcal S$ is of infinite index, so since it is ergodic it is everywhere of infinite index. By the above theorem $\mathcal S$ has a dense orbit in $\Sub(\mathcal R_0)$.
	Note that by \cite[Thm.~10.1]{kechrisSpaceMeasurepreservingEquivalence}, ergodic subequivalence relations actually form a $G_\delta$ set, so the Baire category theorem provides us with a dense $G_\delta$ subset consisting of \emph{ergodic} subequivalence relations of $\mathcal R_0$ satisfying the equivalent conditions of Theorem \ref{thmi: dense orbit}.
\end{example}

Our approach to Theorem \ref{thmi: dense orbit} is very much inspired by a joint work in preparation with Fima, Mukherjee and Patri where we similarly characterize dense orbits in the space of von Neumann subalgebras of the hyperfinite type II$_1$ factor. In this setup, we have to rely on a difficult result of Popa which characterizes the subalgebras which have a sequence of unitary conjugates converging to the trivial subalgebra, see Lem.\ 2.3 from \cite{popaAsymptoticOrthogonalizationSubalgebras2019}. Here, we similarly first need to understand which subequivalence relations have a sequence of translates converging to the trivial subequivalence relations $\Delta_X=\{(x,x)\colon x\in X\}$, which is the content of our next result.

\begin{theoremi}[see Thm.\ \ref{thm: infinite index everywhere characterization}]\label{thmi: infinite index everywhere characterization}
	Let $\mathcal R$ be an aperiodic p.m.p.\ equivalence relation on a standard probability space  $(X,\mu)$. Let $\mathcal S\in\Sub(\mathcal R)$ be a subequivalence relation. The following are equivalent:
	\begin{enumerate}[(i)]
		\item \label{itemi: infinite index} $\mathcal S$ has everywhere infinite index in $\mathcal R$;
		\item \label{itemi: fullgroup orbit} the closure of the $[\mathcal R]$-orbit of $\mathcal S$ contains $\Delta_X$;
		\item \label{itemi: Aut orbit} the closure of the $\Aut(\mathcal R)$-orbit of $\mathcal S$ contains $\Delta_X$.
	\end{enumerate}
\end{theoremi}

The proof of Theorem \ref{thmi: infinite index everywhere characterization} relies on an inductive construction, using a countable dense subgroup of the full group and the fact that such a subgroup acts highly transitively on almost every orbit, a result due to Eisenmann and Glasner \cite[Thm.\ 1.19]{eisenmannGenericIRSFree2016}.

Once this is done, we prove Theorem \ref{thmi: dense orbit} using the fact that finite equivalence relations are dense in $\Sub(\mathcal R_0)$ and that they can always be translated inside any aperiodic subequivalence relation. This allows us to apply Theorem \ref{thmi: infinite index everywhere characterization} on a fundamental domain of the finite subequivalence relation we want to approximate. A more general statement on the closed subspace of hyperfinite subequivalence relations can actually be proved this way, see Corollary \ref{cor: dense orbit in subhyp}.

Using Ioana's intertwining of subequivalence relations \cite[Lem.\ 1.7]{ioanaUniquenessGroupMeasure2012} and the two above examples of subequivalence relations, we moreover have the following result, which again follows from a more general statement on the space hyperfinite subequivalence relations (see Theorem \ref{thm: meager orbits}).

\begin{theoremi}[{see Cor.\ \ref{cor: orbits meager R0}}]\label{thmi: meager orbits}
	Let $\mathcal R_0$ be the ergodic p.m.p. hyperfinite equivalence relation. Then all $[\mathcal R_0]$-orbits are meager in $\Sub(\mathcal R_0)$.
\end{theoremi}

It then follows from Mycielski's theorem (see e.g.\ \cite[Thm.\ 5.3.1]{gaoInvariantDescriptiveSet2009})  that $\mathcal R_0$ contains a continuum of aperiodic subequivalence relations of infinite index which are in pairwise disjoint full group orbits. Moreover, by \cite[Prop.\ 10.1]{kechrisSpaceMeasurepreservingEquivalence} and Example \ref{ex: ergodic example}, such a continuum can be found inside ergodic subequivalence relations.
We leave the following question open. 
\begin{question}
	Can $\Sub(\mathcal R_0)$ contain a comeager $\Aut(\mathcal R_0)$-orbit?
\end{question}

This question can also be asked for equivalence relations $\mathcal R$ different from $\mathcal R_0$, although the existence of dense $\Aut(\mathcal R)$- or $[\mathcal R]$-orbits in their space of subequivalence relations is already an open problem (see the remark preceding Corollary \ref{cor: dense orbit in R0} for some examples where there are no dense orbits though). Also note that if $\mathcal R$ is a type II$_\infty$ or type III ergodic non-singular equivalence relation, it is not hard to show that there are always dense full group orbits in $\Sub(\mathcal R)$ (see the remark right before Section \ref{sec: meager orbits} for a sketch of proof), but the question of comeager orbits is completely open there.

Our paper concludes with a few complexity calculations, answering some questions from an earlier version of \cite{kechrisSpaceMeasurepreservingEquivalence}. Let us highlight here a way of understanding Kechris' uniform topology on the space of subequivalence relations which we use: given two non-singular equivalence relations $\mathcal R_1$ and $\mathcal R_2$, denote by $\mathfrak C(\mathcal R_1,\mathcal R_2)$ the set of Borel subsets $A$ of $X$ such that $\mathcal R_{1\restriction A}=\mathcal R_{2\restriction A}$, and define
$$d_u(\mathcal R_1,\mathcal R_2)=1-\sup_{A\in\mathfrak C(\mathcal R_1,\mathcal R_2)}\mu(A)$$
This defines a metric that we propose to call the uniform metric since it induces the uniform topology defined in \cite[Sec.\ 4.6]{kechrisSpaceMeasurepreservingEquivalence}, although it looks finer a priori (to see that the two topologies coincide, one needs to use \cite[Thm.~1]{ioanaSubequivalenceRelationsPositivedefinite2009}). 

Let us finish by outlining the paper: along with basic preliminaries, Section \ref{sec: preliminaries} provides a first proof of Theorem \ref{thmi: Polish space of subeq} using the $\liminf$.  Section \ref{sec: strong and low} contains the second proof, which relies in a more direct way on the topology of the measure algebra, using as well the lower topology that we define and study therein. Section \ref{sec: dense orbit} is devoted to the proof of Theorem \ref{thmi: dense orbit} and \ref{thmi: meager orbits} in their more general form, and contains in particular the proof of Theorem \ref{thmi: infinite index everywhere characterization} in Subsection \ref{sec: converging to Delta}. Finally, Section \ref{sec: uniform} contains our complexity computations, which use the uniform metric.
\begin{ack}
	I would like to thank Alekos Kechris for his remarks and encouragements on this paper, and Todor Tsankov for discussions around this topic. I am very much indebted to my coauthors Pierre Fima, Kunal Mukherjee and Issan Patri for many of the ideas in this paper, which were developped in the parallel framework of subalgebras of finite von Neumann algebras. I am also thankful to Damien Gaboriau for his comments, and to Stefaan Vaes for allowing me to include his remark on the von Neumann algebraic way of understanding the space of subequivalence relations.
	Finally, I am very grateful to the anonymous referee for their helpful comments which improved the paper in many ways.
\end{ack}

\section{The Polish topology on the space of subequivalence relations in the non-singular setup}\label{sec: preliminaries}

The main purpose of this section is to extend the topology on the space of subequivalence relations from \cite{kechrisSpaceMeasurepreservingEquivalence} to the non-singular setup. To do so, we use the framework of measure algebras, which does yield the right topology for the probability measure-preserving case as noted in \cite[Sec.\ 4.4(1)]{kechrisSpaceMeasurepreservingEquivalence}. This direct approach also 
sheds a new light on some results from \cite{kechrisSpaceMeasurepreservingEquivalence}, such as Theorem 5.1 and Proposition 4.27. 

\subsection{The measure algebra of a standard \texorpdfstring{$\sigma$}{sigma}-finite space}

Let $(Y,\lambda)$ be a standard $\sigma$-finite measured space, i.e.\ a standard Borel space endowed with a $\sigma$-finite atomless measure.
Such spaces are always isomorphic either to $\R$ endowed with the Lebesgue measure or a finite length interval, also endowed with the Lebesgue measure.
The \textbf{measure algebra} of $(Y,\lambda)$ is the space 
$\MAlg(Y,\lambda)$ of all Borel subsets of $X$, where we identify $A,B\subseteq X$ Borel as soon as $\lambda(A\bigtriangleup B)=0$.

In order to endow $\MAlg(Y,\lambda)$ with a complete metric, we first have to choose some finite measure $\mu$ in the same class as $\lambda$, i.e.\ sharing the same measure zero sets. 
Note that this does not change the definition of the measure algebra.
We can then equip the measure algebra $\MAlg(Y,\lambda)=\MAlg(X,\mu)$  with the metric $d_\mu(A,B)=\mu(A\bigtriangleup B)$. 

\begin{lemma}\label{lem: unif cont equiv meas}
	Let $\mu$ and $\mu'$ be two finite measures on $Y$ in the same equivalence class.
	Then $d_\mu$ and $d_{\mu'}$ are uniformly equivalent: given any $\epsilon>0$, there is $\delta>0$ such that for all $A,B\in\MAlg(Y,\mu)$, $d_\mu(A,B)<\delta$ implies $d_{\mu'}(A,B)<\epsilon$ and vice-versa.
\end{lemma}
\begin{proof}
	By symmetry, given $\epsilon>0$ it suffices to find a $\delta$ such that the direct implication holds. 
	But this is a direct consequence of the fact that $\mu$ and $\mu'$ are in the same class: indeed by 
	\cite[Lem.\ 4.2.1]{cohnMeasureTheorySecond2013}, there actually is $\delta>0$ such that whenever $\mu(C)<\delta$, 
	we must have $\mu'(C)<\epsilon$.
\end{proof}

In particular, the above lemma yields that even when $\lambda$ is infinite, its measure algebra is endowed with a natural topology, independent of the choice of some $\mu$ finite in the class of $\lambda$.
The fact that that $d_\mu$ is complete is usually proved via the following lemma, which is important to us on its own.

\begin{lemma}\label{lem: converging subsequence liminf}
	Let $(A_n)$ be a sequence of elements of $\MAlg(Y,\mu)$ such that $\sum_n\mu(A_n\bigtriangleup A_{n+1})<+\infty$. Let $\liminf A_n=\bigcup_{N}\bigcap_{n\geq N} A_n$, then the sequence $(A_n)$ converges to
	$\liminf A_n$.
\end{lemma}
\begin{proof}
	By the Borel-Cantelli lemma, for almost all $x\in X$ there is $N$ such that for all $n\geq N$, $x\not\in A_n\bigtriangleup A_{n+1}$. Let $\epsilon>0$, we then find $N$ and a set $X_0$ such that $\mu(X_0)>1-\epsilon$ and for all $x\in X$ and all $n\geq N$, we have $x\not\in A_n\bigtriangleup A_{n+1}$. Now observe that for all $n\geq N$, we have $A_n\cap X_0=\liminf A_n\cap X_0$, in particular $A_n\bigtriangleup \liminf A_n\subseteq X\setminus X_0$ which has measure less than $\epsilon$ as wanted. 
\end{proof}

\begin{proposition}\label{prop: convergence}
	Let $(A_n)$ be a sequence of elements of $\MAlg(Y,\mu)$ which converges to some $A\in\MAlg(Y,\mu)$. 
	Then we can find a subsequence $(A_{n_k})$ such that 
	$$A=\liminf A_{n_k}=\bigcup_N\bigcap_{k\geq N} A_{n_k}$$
\end{proposition}
\begin{proof}
	Since $(A_n)$ converges, it is Cauchy, in particular it admits a subsequence $(A_{n_k})_k$ such that for all $k\in\N$, $d_\mu(A_{n_k},A_{n_{k+1}})<\frac 1{2^k}$. We thus have by the previous lemma
	$A_{n_k}\to\liminf A_{n_k}=A$.
\end{proof}

\begin{proposition}\label{prop: dmu complete separable}
	The metric $d_\mu$ is complete, and $\MAlg(Y,\mu)$ is separable. 
\end{proposition}
\begin{proof}
	Let $(A_n)$ be a Cauchy sequence, it suffices to show that it has a convergent subsequence. But because the sequence is Cauchy we may find a subsequence $(A_{n_k)})_k$ such that for all $k\in\N$, $\mu(A_{n_k}\bigtriangleup A_{n_{k+1}})<2^{-k}$. Then $\sum_k\mu(A_{n_k}\bigtriangleup A_{n_{k+1}})<+\infty$ and we get the desired result by applying the previous lemma.
	
	Towards proving separability, first note that since $(Y,\mu)$ is standard, we may as well assume it is an interval of finite length endowed with the Lebesgue measure. The regularity of the Lebesgue measure can then be used to show that finite unions of open intervals with rational endpoints are dense in the measure algebra, thus yielding the desired countable dense subset.
\end{proof}

\subsection{The space of subequivalence relations}\label{sec: space of subeq}

Given an action of a countable group $\Gamma$ by Borel bijections on a standard Borel space $X$, we obtain the associated equivalence relation $\mathcal R_\Gamma=\{(x,\gamma x)\colon x\in X, \gamma\in\Gamma\}$ whose equivalence classes are exactly the $\Gamma$-orbits.
This equivalence relation is a Borel subset of $X\times X$, and its classes are countable: $\mathcal R_\Gamma$ is by definition a \textbf{countable Borel equivalence relation} (\textbf{CBER}, which can conveniently be read as “seeber").
Conversely, the Feldman-Moore theorem ensures us that every CBER comes from a Borel action of a countable group, see for instance \cite[Thm.\ 1.3]{kechrisTopicsOrbitEquivalence2004}. 

Given such an equivalence relation, define its \textbf{Borel pseudo full group} as the set of all partial Borel bijection $\varphi: \dom\varphi\subseteq X\to \rng\varphi\subseteq X$ such that for all $x\in A$, we have $(x,\varphi(x))\in\mathcal R$. By definition, the \textbf{Borel full group} is the group made of elements of the Borel pseudo full group whose domain and range is equal to $X$.

If $(X,\mu)$ is now a standard probability space, we say that a CBER is \textbf{non-singular} (or that $\mu$ is quasi-invariant, or that $\mathcal R$ is null-preserving) if for all $\varphi$ in its Borel pseudo full group,
we have $\mu(\dom\varphi)=0$ iff $\mu(\rng \varphi)=0$.
The main (and only by the Feldman-Moore theorem) source of such equivalence relations is the following:
\begin{definition}
An action of a countable group $\Gamma$ by Borel bijections on $X$ is called \textbf{non-singular} (or null-preserving, or quasi-preserving $\mu$) if for all $A\subseteq X$ Borel, we have $\mu(A)=0$ iff $\mu(\gamma A)=0$.
\end{definition}

\begin{remark}
	By Lemma \ref{lem: unif cont equiv meas}, non-singular actions yield action by uniformly continuous homeomorphisms on the measure algebra $\MAlg(X,\mu)$.
\end{remark}

Given $\Gamma\act (X,\mu)$ non-singular the associated CBER $\mathcal R_\Gamma$ is non-singular, and conversely any non-singular CBER arises in this manner (see \cite[Prop.\ 8.1]{kechrisTopicsOrbitEquivalence2004}).

Now let $\mathcal R$ be a non-singular CBER.
We equip $\mathcal R$ with the left measure $M$ defined by 
$$M(A)=\int_X\abs{A_x}d\mu(x)$$
for all $A\subseteq \mathcal R$ Borel. Let $m$ be an equivalent probability measure, then the measure algebra of $\mathcal R$ with respect to $M$ is equal to that of $\mathcal R$ with respect to $m$.
This measure is $\sigma$-finite since $\mathcal R$ can be covered by the graphs of the elements of a countable group $\Gamma$ such that $\mathcal R=\mathcal R_\Gamma$ (each such graph has measure $1$).

It then follows from Proposition \ref{prop: dmu complete separable}  that  $\MAlg(\mathcal R,M)$ is a Polish space. Moreover by Lemma \ref{lem: unif cont equiv meas} the topology (and actually the uniform structure) do not depend on the choice of $m$. Since the measure $m$ is not canonical, we will most of the time write our topological space as $\MAlg(\mathcal R,M)$. 

A Borel subset of $\mathcal R$ is called a Borel subequivalence relation if it is an equivalence relation on $X$. 
An element of $\MAlg(\mathcal R,M)$ is called a subequivalence relation if its equivalence class \emph{contains} an equivalence relation. Denote by $\Sub(\mathcal R)\subseteq \MAlg(\mathcal R,M)$ the space of subequivalence relations of $\mathcal R$. We can now easily prove the following result of Kechris (our proof is moreover motivated by \cite[Thm.\ 5.1]{kechrisSpaceMeasurepreservingEquivalence}, which we have essentially recast as Proposition \ref{prop: convergence}).

\begin{theorem}[{Kechris, \cite[Sec.\ 4.4.(1)]{kechrisSpaceMeasurepreservingEquivalence}}]\label{thm: sub R is Polish}
	The space $\Sub(\mathcal R)$ of subequivalence relations of $\mathcal R$ is a closed subset of $\MAlg(\mathcal R,M)$, hence Polish.
\end{theorem}
\begin{proof}
	Observe that the $\liminf$ of any sequence of equivalence relations is an equivalence relation. By Proposition \ref{prop: convergence} the result follows.
\end{proof}

Observe that we did not even use that $\mathcal R$ was non-singular in the above. However,  we need non-singularity for our space to encode the right notion of equality of subequivalence relations as follows. 

\begin{proposition}
	Let $\mathcal S_1$, $\mathcal S_2$ be two Borel subequivalence relations of a non-singular CBER $\mathcal R$, suppose $M(\mathcal S_1\bigtriangleup \mathcal S_2)=0$. Then there is a full measure $\mathcal S_1$ and $\mathcal S_2$-invariant set onto which $\mathcal S_1$ and $\mathcal S_2$ are the same. 
\end{proposition}
\begin{proof}
	By definition for almost all $x$ we have $[x]_{\mathcal S_1}=[x]_{\mathcal S_2}$. Since $\mathcal S_1$ is non singular we find a smaller $\mathcal S_1$-invariant set $X_0$ for which the same conclusion holds (if $\mathcal S_1=\mathcal R_\Gamma$ we let $\displaystyle X_0=\bigcap_{\gamma\in \Gamma} \gamma X_1$, where $X_1=\{ x\in X\colon [x]_{\mathcal S_1}=[x]_{\mathcal S_2}\}$). This finishes the proof. 
\end{proof}

We will also see in the next section that non-singularity is important if we want to concretely see our space of subequivalence relations as a closed set, i.e. to say that the axioms of subequivalence relations each define closed sets. Let us conclude this section by pointing out yet another way of understanding the topology which was pointed out to us by Stefaan Vaes.

\begin{remark}
	Given a non-singular equivalence relation $\mathcal R$ on $(X,\mu)$, consider the space of subalgebras of its von Neumann algebra $L\mathcal R$, which is Polish when endowed with the Maréchal topology (\cite{marechalTopologieStructureBorelienne1973}, see also \cite{fimaMichaelsSelectionTheorem2024} for a detailed proof).
	Then it is a result of Aoi that subequivalence relations of $\mathcal R$ are in one-to-one correspondence
	with subalgebras of $L\mathcal R$ containing $\LL^\infty(X,\mu)$ \cite{aoiConstructionEquivalenceSubrelations2003}.
	It is not hard to see that such subalgebras form a closed set, and the map which associates to $\mathcal S\in\Sub(\mathcal R)$ the subalgebra $L\mathcal S$ of $L\mathcal R$ is a homeomorphism 
	(see Theorem E from the recent preprint by Shuoxing Zhou \cite{zhouNoncommutativeTopologicalBoundaries2024}), yielding yet another way of understanding the Polish topology
	induced by the measure algebra of $\mathcal R$ on $\Sub(\mathcal R)$.
	For (many!) other interpretations of this topology in the probability measure preserving setup, the reader should consult \cite{kechrisSpaceMeasurepreservingEquivalence}.
\end{remark}

\section{Strong and lower topologies}\label{sec: strong and low}

In this section, we first study the strong topology on the space of subequivalence relations as defined in \cite{kechrisSpaceMeasurepreservingEquivalence}, which will be useful towards establishing our orbit density result (Theorem \ref{thmi: dense orbit}). 
We then use a lower topology on the measure algebra so as to obtain a natural proof of the fact that the space of subequivalence relations is closed in the measure algebra, and of the fact that the map which takes a subgraph to the equivalence it generates is Baire class one \cite[Prop.\ 19.1]{kechrisSpaceMeasurepreservingEquivalence}.

\subsection{Topologies on measure algebras of \texorpdfstring{$\sigma$}{sigma}-finite spaces}

We first go back to measure algebras of $\sigma$-finite spaces, which we now write as $(Y,M)$, motivated by our previous example where $Y$ is a non-singular CBER.

\begin{proposition}\label{prop: exhaustion by finite and topology}
	Let $(Y,M)$ be a standard $\sigma$-finite space. 
	Let $(Y_n)$ be a countable family of finite measure subsets of $Y$ such that $\bigcup_nY_n=Y$. Then the map $\Phi:\MAlg(Y,M)\to\prod_n \MAlg(Y_n,M_{Y_n})$ which takes $A$ to the sequence $(A\cap Y_n)$ is a homeomorphism onto its image. 
\end{proposition}
\begin{proof}
	Observe that the probability measure 
	$$m(A):=\sum_n \frac 1{2^{n+1}M(Y_n)}M(A\cap Y_n)$$
	is equivalent to $M$ and thus the associated metric $d_m$ induces the topology of $\MAlg(Y,M)$. 	
	Now the space $\prod_n \MAlg(Y_n,M_{Y_n})$ can be endowed with a compatible metric $\tilde d$ given by 
	$$\tilde d((A_n),(B_n))=\sum_n \frac 1{2^{n+1}M(Y_n)}M(A_n\bigtriangleup B_n),$$
	It is then straightforward to check that $\Phi$ is an isometry for the metric $d_m$ on its domain and $\tilde d$ on its range.	
\end{proof}

We now go on to obtain a more precise version of this result, following Kechris' approach to the \emph{strong topology} on $\Sub(\mathcal R)$.

Observe that every element $\varphi$ of the Borel pseudo full group of a CBER $\mathcal R$ is completely determined by its \textbf{graph}, namely the set $\{(x,\varphi(x))\colon x\in \dom\varphi\}$, which is a Borel subset of $\mathcal R$.
From now on, we identify elements of the Borel pseudo full group to their graph. Doing this up to measure zero when $\mathcal R$ is nonsingular, we arrive at the following definition.

\begin{definition}
	The \textbf{pseudo full group} of a non-singular equivalence relation $\mathcal R$ on $(X,\mu)$ is the quotient of the Borel pseudo full group by the equivalence relation which identifies $\varphi_1$, $\varphi_2$ when $M(\varphi_1\bigtriangleup \varphi_2)=0$. 
	It is denoted by $[[\mathcal R]]$, and for ease of notation we still write its elements as $\varphi$'s.
	Finally the \textbf{full group} of $\mathcal R$ denotes all the elements of $[[\mathcal R]]$ with full measure domain and range.
\end{definition}

It is a direct consequence of the definition of $M$ that $M(\varphi_1\bigtriangleup \varphi_2)=0$ iff $\mu(\dom\varphi_1\bigtriangleup\dom\varphi_2)=0$ and for almost all $x\in \dom\varphi_1\cap\dom\varphi_2$, we have $\varphi_1(x)=\varphi_2(x)$. 
The Borel pseudo full group is stable under countable increasing unions and arbitrary countable intersections, so it is stable under taking $\liminf$. Following the same approach as for the proof of Theorem \ref{thm: sub R is Polish}, one obtains that $[[\mathcal R]]$ is a closed subset of $\MAlg(\mathcal R,M)$, in particular it is Polish. Note however that $[\mathcal R]$ is not closed for this topology (see Example \ref{ex: full group leaving}).

We now make a slight modification of Kechris' uniquely generating sequences of involutions.
Let us say that an element $\varphi$ of the Borel pseudo-full group is a \textbf{moving partial  involution} if
$\dom\varphi=\rng\varphi$ and for all $x\in\dom\varphi$, we both have $\varphi(x)\neq x$ and $\varphi(\varphi(x))=x$.

\begin{definition}
	A sequence $(\varphi_n)_{n\in\N}$ of moving partial involutions of the Borel pseudo-full group of a CBER $\mathcal R$  is called a \textbf{uniquely generating sequence of moving partial involutions} if 
	\[
	\mathcal R\setminus \Delta_X=\bigsqcup_{n\in\N} \varphi_n,
	\]
	where $\Delta_X=\{(x,x)\colon x\in X\}$ is the equality relation.
\end{definition}

\begin{proposition}\label{prop: unique moving invol}
	Every CBER $\mathcal R$ admits a uniquely generating sequence of moving partial involutions.
\end{proposition}
\begin{proof}
	Let $(T_n)_{n\in\N}$ be a uniquely generating sequence of involutions for $\mathcal R$ as in \cite[Prop.\ 4.6]{kechrisSpaceMeasurepreservingEquivalence}, and for every $n\in\N$ let $\varphi_n$ be the restriction of $T_n$ to its support. 
\end{proof}
\begin{remark}
	The disjointness of the graphs of the $\varphi_n$'s simplifies a bit the arguments we carry out in this paper since it guarantees that all the $\varphi_i(x)$ are distinct whenever they are defined (see in particular the proof of \eqref{item: infinite index}$\Rightarrow$\ref{item: fullgroup orbit} in Theorem \ref{thm: infinite index everywhere characterization}). The disjointness also allows us to upgrade Proposition \ref{prop: exhaustion by finite and topology} and obtain that the map 
	\[
	\MAlg(\mathcal R,M)\to \MAlg(\Delta_X,M_{\Delta_X})\times \prod_{n\in\N} \MAlg(\varphi_n, M_{\varphi_n}))
	\]
	which takes $A$ to the sequence $(A\cap Y_n)$ is not only a homeomorphism onto its image, but actually surjective.
\end{remark}
The following corollary is a slight reformulation of \cite[Thm.\ 4.13]{kechrisSpaceMeasurepreservingEquivalence} which identifies the strong and the weak topology on the space of subequivalence relations.  
\begin{corollary}\label{cor: strong topology}
	Let $\mathcal R$ be a non-singular equivalence relation, let $(\varphi_n)$ be a uniquely generating sequence of moving partial involutions. Then the map 
		\[
	\Sub(\mathcal R)\to \prod_{n\in\N} \MAlg(\varphi_n, M_{\varphi_n})
	\]
	which takes $\mathcal S$ to the sequence $(\mathcal S\cap \varphi_n)_{n\in\N}$ is a homeomorphism onto its image. 
\end{corollary}
\begin{proof}
	This is a direct consequence of the definition of uniquely generating sequences of moving partial involutions and of Proposition \ref{prop: exhaustion by finite and topology}
	once we note that every equivalence relation contains $\Delta_X$.
\end{proof}

Following the above corollary and Kechris' terminology, we now call \textbf{strong topology} the topology induced by $d_m$ on $\Sub(\mathcal R)$, which is Polish by Theorem \ref{thm: sub R is Polish}.

\subsection{Lower topologies on measure algebras}

Let us first work a bit more on the relationship between $M$ and $m$. Observe that $M$ defines a lower-semi continuous function on $\MAlg(Y,M)$ (possibly taking $+\infty$ as a value). Indeed if $M(A)>\alpha$, we have a $\delta>0$ such that $M(C)<M(A)-\alpha$ whenever $C$ is a subset of $A$ satisfying $m(C)<\delta$, and so as soon as $m(A\setminus B)<\delta$ we have $M(B)>\alpha$. In particular for every $k\in\N$ the set of elements with measure $\leq k$ is closed. 

Let us denote by $\MAlg_f(\mathcal R,M)$ the $F_\sigma$ subset of elements of $\MAlg(\mathcal R,M)$ with finite $M$-measure.
It can be endowed with the metric $d_M(A,B)=M(A\bigtriangleup B)$ which clearly refines the topology induced by $d_m$.
Moreover, the proof of Proposition \ref{prop: dmu complete separable} can easily be adapted to show that $d_M$ is complete separable (see \cite[Lem.\ 2.1]{lemaitrePolishTopologiesGroups2022} for a complete proof). Let us see why $d_m$ and $d_M$ yield different topologies on $\MAlg_f(\mathcal R,M)$.

\begin{example}\label{ex: full group leaving}
Going back to the case $Y=\mathcal R$, observe that if $T$ in the full group of $\mathcal R$ is aperiodic, then $T^n\to \emptyset$ for $d_m$ while it stays in the sphere around $1$ for $d_M$. In particular $d_m$ does not refine $d_M$. 
Let us also mention that $[\mathcal R]$ is $d_M$-closed, and that the metric induced by $d_M$ on $[\mathcal R]$ is the well-known uniform metric $d_u(S,T)=\mu(\{x\in X\colon S(x)\neq T(x)\})$, endowing $[\mathcal R]$ with a Polish group topology.
\end{example}
Despite the above example, $M$ and $m$ share the same \emph{lower topologies} when restricted to $\MAlg_f(Y,M)$. These lower topologies are not Hausdorff, however they have nice continuity properties.
\begin{definition}
	Let $(Y,M)$ be a $\sigma$-finite measured space. We endow its measure algebra with the \textbf{lower $M$-topology} which is defined by taking as neighborhoods of an element $A\in\MAlg(Y,M)$ all sets of the form $$\{B\in\MAlg(Y,M): M(A\setminus B)<\epsilon\}$$ for some $\epsilon>0$. 
\end{definition}
This neighborhood system does satisfy the axiom that implies that it is a basic neighborhood system for the topology it induces, namely every neighborhood $V$ contains a smaller neighborhood all of whose elements admit $V$ as a neighborhood (axiom $\mathrm{(V_{IV})}$ in \cite[Chap.\ 1, \S 1.2]{bourbakiGeneralTopologyChapters1998}). This is simply because if $M(A\setminus B)<\frac{\epsilon}2$ and $M(B\setminus C)<\frac\epsilon 2$ then $M(A\setminus C)<\epsilon$ ).

Note that the only neighborhood of the emptyset is the whole measure algebra, while $Y$ is contained in every neighborhood of every point.
Moreover, $\{\emptyset\}
$ is closed, and more generally $M$ is lower semi-continuous for the lower $M$-topology.

\begin{lemma}\label{lem: lower topology on finite measure}
	If $M$ is a $\sigma$-finite measure on $Y$ and $m$ is an equivalent finite measure, then the measures $M$ and $m$ induce the same lower topology on the space $\MAlg_f(Y,M)$ of finite $M$-measure subsets.
\end{lemma}
\begin{proof}
	Take $A$ such that $M(A)<+\infty$. 
	Because $M$ and $m$, when restricted to $A$, are equivalent, we find $\delta$ such that every subset of $A$ of $m$ measure less than $\delta$ has $M$-measure less than $\epsilon$. Then $\{B: m(A\setminus B)<\delta\}\subseteq\{B: M(A\setminus B)<\epsilon\}$. So the lower $m$-topology refines the lower $M$-topology. The other inequality works the same. 
\end{proof}
Let us note that the space of infinite measure subsets is open for the lower $M$-topology, while it is not for the lower $m$-topology.
Indeed if $A$ has infinite measure, the set of $B$ such that $M(A\setminus B)<1$ does not contain any $m$-neighborhood, because there are subsets of $A$ of very small $m$-measure but still with infinite measure.

We will now list various interesting properties of the lower topology, and then relate it to the usual topology in the last two lemmas.

\begin{lemma}\label{lem: intersection is lower continous}
	Let $M$ be a $\sigma$-finite measure on $Y$. Then the intersection map $(A_1,A_2)\mapsto A_1\cap A_2$ is continuous, where we put the $M$-lower topology everywhere.
\end{lemma}
\begin{proof}
	Let $\epsilon>0$. Note that $(A_1\cap A_2)\setminus (B_1\cap B_2)\subseteq (A_1\setminus B_1)\cup (A_2\setminus B_2)$, so if
	 if $M(A_1\setminus B_1)<\epsilon$ and $M(A_2\setminus B_2)<\epsilon$, then $M((A_1\cap A_2)\setminus (B_1\cap B_2))<2\epsilon$ as wanted.
\end{proof}

\begin{lemma}\label{lem: countable union is lower continous}
	Let $m$ be a finite measure on $Y$. Then the countable union map $(A_n)_{n\in\N}\mapsto \bigcup_{n\in\N}A_n$ is continuous for the $m$-lower topology everywhere.
\end{lemma}
\begin{proof}
	Let $(A_n)_{n\in\N}$ be a family of elements of the measure algebra, then since $m$ is finite we may find $N$ such that $\mu((\bigcup_n A_n)\setminus\bigcup_{n<N} A_n))<\epsilon$. Now if we take $B_0,...,B_{N-1}$ such that $m(A_i\setminus B_i)<\epsilon/N$, and any $B_{N},B_{N+1},...$, we see that $m((\bigcup_n A_n)\setminus (\bigcup_n B_n))<2\epsilon$ as wanted.
\end{proof}
\begin{remark}
	Note that the countable union map is \emph{not} continuous if we put the usual topology everywhere (indeed $X$ can easily be obtained as a limit of unions of sequences of sets which  converge coordinatewise to $\emptyset$). Here is nevertheless a positive result for the usual topology which will be useful later on. 
\end{remark}
\begin{lemma}\label{lem: continuous union on partition}
	Let $m$ be a finite measure on $Y$, let $(Y_j)_{j\in\N}$ be a disjoint family of measurable subsets of $Y$. Then the restriction of the countable union map to the set
	\[
	\{(A_n)\in\MAlg(Y,m)^\N\colon \forall n\in\N, A_n\subseteq Y_n\},
	\]
	is continuous for the $d_m$-topology.
\end{lemma}
\begin{proof}
	Since the $Y_n$'s are disjoint and $m$ is finite, given any $\epsilon>0$ we can find $N$ such that $m(\bigcup_{n>N} Y_n)<\epsilon$, in particular if $(A_n)$ satisfies $A_n\subseteq Y_n$ for every $n$ then its union is up to an $\epsilon$ error independent of $(A_n)_{n>N}$. The conclusion thus follows from the continuity of finite unions for the $d_m$-topology.
\end{proof}

We now relate the Hausdorff topology associated to the metric $d_m$ and to the corresponding lower $m$-topology. 
\begin{lemma}\label{lem: baire class one lower}
	Let $m$ be a finite measure on $Y$, denote by $\tau_m^l$ the lower $m$-topology and by $\tau_m$ the topology induced by the metric $d_m$.
	Then the identity map $(\MAlg(Y,m),\tau^m_l)\to (\MAlg(Y,m),\tau_m)$ is Baire class one.
\end{lemma}
\begin{proof}
	We begin by noting some useful continuity properties for the lower $m$-topology.
	By definition, for every $A\in\MAlg(Y,m)$ the map $B\mapsto m(A\setminus B)$ is upper semi-continuous for the lower topology. 
	Since $m(B)=m(X)-m(X\setminus B)$, we deduce that $m$ defines a lower semi-continuous map $(\MAlg(Y,m),\tau_m^l)\to \R^+$.
	By Lemma \ref{lem: intersection is lower continous}, for every $A\in\MAlg(Y,m)$ the map $B\mapsto m(A\cap B)$ is thus lower semi-continuous as well. 
	
	We finally observe that
	$$d_m(A,B)=m(A\bigtriangleup B)=m(A)+m(B)-2 m(A\cap B),$$
	so $d_m(A,\cdot)$ is an Baire class one function on $(\MAlg(Y,\nu),\tau_m^l)\to \R^+$, being the sum of two Baire class one functions. In particular, $d_m$-open balls are $F_\sigma$ for $\tau_m^l$.
	Since every $\tau_m$-open set is a countable union of open balls by separability, the conclusion follows.
\end{proof}

We conclude this section by a lemma which interwines the two topologies we have defined in presence of a finite measure.

\begin{lemma}\label{lem: inclusion is closed}
	Let $m$ be a finite measure on $Y$.
	The space of all $(A,B)$ such that $A\subseteq B$ is closed if we put on the first coordinate the $m$-lower topology, and on the second the usual measure algebra topology. 
\end{lemma}
\begin{proof}
	Suppose that $A\not\subseteq B$, let $\epsilon=m(A\setminus B)$. Then if $m(A\setminus A')<\epsilon/2$ and $m(B'\bigtriangleup B)<\epsilon/2$, we still have $A'\not\subseteq B'$. 
\end{proof}

\subsection{Seeing that the space of subequivalence relations is closed}

We now provide a more concrete proof that the space of subequivalence relations is Polish by seeing that the conditions which make an equivalence relation define closed sets in $\MAlg(\mathcal R,M)$ for the topology induced by $d_m$. First note that reflexivity means containing $\Delta_X$, and hence defines a closed set by Lemma \ref{lem: inclusion is closed}. 

Let us fix a non-singular equivalence relation $\mathcal R$, we recall that $m$ is some finite measure equivalent to the measure $M$ obtained by integrating the cardinalities of vertical fibers. 
Observe that the flip $\sigma:(x,y)\mapsto (y,x)$ quasi-preserves $M$ (this is one characterization of non-singularity, see \cite[Prop.\ 8.2]{kechrisTopicsOrbitEquivalence2004}), in particular it quasi-preserves $m$ and hence induces a homeomorphism on $\MAlg(\mathcal R,M)$. 

This implies that the space of (Borel) \textbf{subgraphs} of $\mathcal R$ (symmetric subsets of $\mathcal R\setminus\id_X$) is closed in $\MAlg(\mathcal R,M)$, because it is the set of fixed points of $\sigma$
contained in $\mathcal R\setminus\id_X$. In particular, we recover the following fact from \cite[Chap.\ 18]{kechrisSpaceMeasurepreservingEquivalence}.

\begin{proposition}\label{prop:graphing is polish}
	The space of subgraphs is Polish for the topology induced by the measure algebra $\MAlg(\mathcal R,M)$.\qed
\end{proposition}

\begin{remark}
	This could also be seen through the same approach as for Theorem \ref{thm: sub R is Polish}, noting that the liminf of any sequence of subgraphs is a subgraph.
\end{remark}

To see that the space of subequivalence is closed, we need to deal with transitivity. Observe that on the space of measurable subsets of $\mathcal R$, we have an operation 
$\circ$ defined by $A\circ B:=\{(x,z): \exists y\in X: (x,y)\in A\text{ and }(y,z)\in B\}$. Note that $\circ$ extends the composition of elements of the pseudo full group of $\mathcal R$, and that it is well defined only because we are in the non-singular setup. 

\begin{lemma}
	On the pseudo full group of $\mathcal R$, the composition is continuous for the $M$-lower topology, and hence for the $m$-lower topology.
\end{lemma}
\begin{proof}
	 Note that the $M$-measure of the graph of an element of the pseudo full group is the $\mu$-measure of its domain, so by Lemma \ref{lem: lower topology on finite measure}, we only need to check the continuity for the lower $M$-topology. 

	Now fix $\varphi_0,\psi_0\in[[\mathcal R]]$, let $\epsilon>0$, find $\delta$ such that whenever $\mu(A)<\delta$ we have $\mu(\psi_0\inv(A))<\epsilon/2$. Take $\varphi$ such that $M(\varphi\setminus\varphi_0)<\delta$, then $\mu(\{x\in \dom(\varphi_0): \varphi(x)\neq\varphi_0(x)\})<\delta$ (in the definition we consider that $\varphi_0(x)$ is different from $\varphi(x)$ if $x\not\in\dom\varphi$). Then take $\psi$ such that $M(\psi\setminus\psi_0)<\epsilon/2$, the set $X_1=\{x\in\dom\psi_0: \psi(x)=\psi_0(x)\}$ contains $\dom\psi_0$ up to an $\epsilon/2$ error, and $$\psi_{\restriction X_1}\inv(\{x\in \dom(\varphi_0): \varphi(x)\neq\varphi_0(x)\})=\psi_{0\restriction X_1}\inv(\{x\in \dom(\varphi_0): \varphi(x)\neq\varphi_0(x)\})$$
	has thus measure at most $\epsilon/2$.
	We conclude that the set of $x\in \dom \varphi_0\psi_0$ such that $\varphi\psi(x)=\varphi_0\psi_0(x)$ contains $\dom\varphi_0\psi_0$ up to an $\epsilon/2+\epsilon/2=\epsilon$ error as wanted. 
	The fact  that the continuity also holds for the $m$-lower topology is then a direct consequence of Lemma \ref{lem: lower topology on finite measure} since elements of $[[\mathcal R]]$ have $M$-measure at most $1$.
\end{proof}

\begin{lemma}\label{lem: composition continuous}
	The map $A,B\mapsto A\circ B$ is continuous if we put everywhere the lower $m$-topology.
\end{lemma}
\begin{proof}
	Fixing a countable subgroup $\Gamma\leq[\mathcal R]$ such that $\mathcal R=\mathcal R_\Gamma$,
	the desired continuity follows from the previous lemma along with the formula $$A\circ B=\bigcup_{\gamma,\gamma'\in\Gamma}\gamma\cap A\circ (\gamma'\cap B)$$ and the fact that finite intersections and countable unions are continuous maps for the lower topology (Lemma \ref{lem: intersection is lower continous} and \ref{lem: countable union is lower continous}).
\end{proof}

We thus arrive at the desired result saying that the space of transitive subrelations is closed. 

\begin{proposition}
	The space of $A\subseteq\mathcal R$ such that $A\circ A\subseteq A$ is closed in the Polish topology on $\MAlg(\mathcal R,M)$.
\end{proposition}
\begin{proof}
	This follows directly from the previous lemma along with Lem. \ref{lem: inclusion is closed}.
\end{proof}

\begin{theorem}
	$\Sub(\mathcal R)$ is a closed subset of $\MAlg(\mathcal R,M)$ for the $d_m$-topology.
\end{theorem}
\begin{proof}
	We already observed that being reflexive exactly means containing $\Delta_X$ and thus is a closed condition by Lemma \ref{lem: inclusion is closed}.
	The fact that symmetry is a closed condition is a reformulation of the fact that subgraphs form a closed set (Proposition \ref{prop:graphing is polish}). 
	Finally transitivity is a closed condition by the previous proposition. 
	So $\Sub(\mathcal R)$ is the intersection of three closed subsets, hence closed itself.
\end{proof}

We also have the following nice consequence of our work on the lower topology, which generalizes Proposition 19.1 from \cite{kechrisSpaceMeasurepreservingEquivalence}.

\begin{proposition} \label{prop: continuity of generating}
	The map which takes a measurable subset of $\mathcal R$ to the equivalence relation it generates is continuous in the lower topology, in particular it is Baire class one for the strong topology (induced by $d_m$). 
\end{proposition}
\begin{proof}
	This map associates to a $A\subseteq \mathcal R$ the equivalence relation $\bigcup_{n\in\N} (\Delta_X\cup A\cup \sigma(A))^{\circ n}$, 
	so it is a continuous map by the continuity of $\sigma$,  of countable unions (Lemma \ref{lem: countable union is lower continous}) 
	and of composition (Lemma \ref{lem: composition continuous}). 
	The statement about the strong topology then follows directly from Lemma \ref{lem: baire class one lower}.
\end{proof}

We can finally obtain the following strengthening of Proposition 4.27 from \cite{kechrisSpaceMeasurepreservingEquivalence}.

\begin{corollary}
	The map which takes a sequence $(\mathcal R_n)$ of subequivalence relations of $\mathcal R$ to the equivalence relation $\bigvee_n \mathcal R_n$ they generate is continuous in the lower topology, in particular it is Baire class one for the strong topology.
\end{corollary}
\begin{proof}
	This is a direct consequence of the previous result along with the fact that taking countable unions is continuous in the lower topology  (Lemma \ref{lem: countable union is lower continous}).
\end{proof}

\section{Dense orbits for hyperfinite subequivalence relations in the p.m.p.\ case}\label{sec: dense orbit}

In the present section, we exclusively work in the probability measure preserving setup, so $\mathcal R$ is a p.m.p.\ equivalence relation on $(X,\mu)$.
Recall that $\Aut(\mathcal R)$ is the group of all $T\in\Aut(X,\mu)$ such that $T\times T\;(\mathcal R)=\mathcal R$. 
Observe that $\Aut(\mathcal R)$ acts on $\MAlg(\mathcal R,M)$ by $T\cdot (x,y)= (T(x),T(y))$ and preserves the measure $M$. 
The group $\Aut(\mathcal R)$ contains the full group of $\mathcal R$, and it is a Polish group for the topology induced by the group $\Aut(\mathcal R,M)$ of all measure-preserving bijections of $(\mathcal R,M)$ (see \cite[Prop.\ 6.3]{kechrisGlobalaspectsergodic2010} where $\Aut(\mathcal R)$ is denoted $N[\mathcal R]$). The trivial subequivalence relation will play an important role; we denote it by 
$\Delta_X=\{(x,x)\colon x\in X\}$ (given our identification of full groups elements to their graphs, we could also write it as the identity map $\id_X$).

\subsection{Preliminaries on conditional measures and index}\label{sec: cond meas and index}

In order to work in the p.m.p.\ setup, we have to understand which sets can be taken to other sets by (pseudo) full group elements up to measure zero.
This is done here through the concept of \emph{conditional measures}, a low-tech version of the ergodic decomposition which was already present in Dye's founding paper \cite{dyeGroupsMeasurePreserving1959}. We already used them in a previous paper on non-ergodic p.m.p.\ equivalence relations, see \cite[Sec.\ 2]{lemaitrefullgroupsnonergodic2016}, but we recently gave a more detailed exposition for general full groups in a joint work with Slutsky  which we use here as a reference (see Appendix D in \cite{lemaitreL1FullGroups2025}).

\begin{definition}
	Let $\mathcal R$ be a p.m.p.\ equivalence relation on $(X,\mu)$, denote by $M_{\mathcal R}$ the closed subalgebra of $(X,\mu)$ consisting of $[\mathcal R]$-invariant subsets,
	and by $\mathbb E_{\mathcal R}$ the projection  $\LL^2(X,\mu)\to\LL^2(X,M_{\mathcal R},\mu)$.
	Given $A\in\MAlg(X,\mu)$, its \textbf{$M_{\mathcal R}$-conditional measure} $\mu_{\mathcal R}$ is the $M_{\mathcal R}$-measurable function
	\[
	\mu_{\mathcal R}(A)=\mathbb E_{M_{\mathcal R}}(\chi_A).
	\]
\end{definition}

It can be checked that $\mu_{\mathcal R}$ takes values in $[0,1]$, satisfies the usual axioms of a measure, and that elements of $[\mathcal R]$ preserve $\mu_{\mathcal R}$ (see \cite[Prop.\ D.6]{lemaitreL1FullGroups2025}).
Say that $\mathcal R$ is \textbf{ergodic} when $M_{\mathcal R}=\{\emptyset,X\}$, then $\mathcal R$ is ergodic iff $\mu_{\mathcal R}=\mu$.
By \cite[Prop.\ D.10]{lemaitreL1FullGroups2025}, we moreover have:
\begin{lemma}\label{lem: full group orbit on malg}
	Let $A,B\in\MAlg(X,\mu)$, let $\mathcal R$ be a p.m.p.\ equivalence relation on $(X,\mu)$. The following are equivalent:
	\begin{enumerate}[(i)]
		\item there is $\varphi\in[[\mathcal R]]$ such that $\dom\varphi=A$ and $\rng\varphi=B$;
		\item $\mu_{\mathcal R}(A)=\mu_{\mathcal R}(B)$.
	\end{enumerate}
\end{lemma}

We finally note the following important consequence of aperiodicity (having only infinite classes) for a p.m.p.\ equivalence relation (it is actually a characterization, but we don't need that).

\begin{proposition}[{Maharam's Lemma, see \cite[Thm.\ D.12]{lemaitreL1FullGroups2025}}]
	\label{prop: maharam}
	Let $\mathcal R$ be a p.m.p.\ aperiodic equivalence relation, let $A\in\MAlg(X,\mu)$, let $f:X\to[0,1]$ be an $M_{\mathcal R}$-measurable function such that $0\leq f\leq \mu_{\mathcal R}(A)$. Then there is $B\subseteq A$ such that $\mu_{\mathcal R}(B)=f$.
\end{proposition} 

We finish this section by recalling some important definitions on the index of subequivalence relations.

\begin{definition}
Let $\mathcal R$ be a non-singular equivalence relation on $(X,\mu)$. We say that $\mathcal S\in\Sub(\mathcal R)$
\begin{itemize}
	\item has \textbf{infinite index} in $\mathcal R$ if for almost all $x\in X$, 
	the $\mathcal R$-equivalence class of $x$ contains infinitely many distinct $\mathcal S$-classes.
	\item has \textbf{finite index} in $\mathcal R$ if for almost all $x\in X$, the $\mathcal R$-equivalence class of $x$ is the union of a finite set of $\mathcal S$-classes.
	\item has \textbf{everywhere infinite index} in $\mathcal R$ if for all $A\subseteq X$ such that $\mu(A)>0$, the restriction of $\mathcal S$ to $A$ has infinite index in the restriction of $\mathcal R$ to $A$.
\end{itemize}
\end{definition}

\begin{remark}
	It might not be clear at first sight that having infinite index is not the same as having everywhere infinite index.
	Here is the simplest example of an infinite index $\mathcal S\in\Sub(\mathcal R)$, not everywhere of infinite index: take $\mathcal R$ ergodic, let $A\subseteq X$ of positive non full measure, and let $\mathcal S=\mathcal R_{\restriction A}\sqcup \Delta_{X\setminus A}$. 
\end{remark}

Let us note that when $\mathcal R$ is ergodic, since the number of $\mathcal S$-classes inside the $\mathcal R$-class of $x\in X$ is $\mathcal R$-invariant, it is constant almost everywhere, and hence $\mathcal S$ either has finite or infinite index in $\mathcal R$. Similarly if $\mathcal S$ is ergodic and of infinite index, then it is everywhere of infinite index because the $\mathcal S$-class of almost every $x\in X$ intersects $A$. We finally note the following non-ergodic way of producing everywhere infinite index subequivalence relations.

\begin{lemma}
	Let $\mathcal R_1$ and $\mathcal R_2$ be non-singular aperiodic equivalence relations on respective standard probability spaces $(X_1,\mu_1)$ and $(X_2,\mu_2)$. 
	Then $\mathcal R_1\times \Delta_{X_2}$ is everywhere of infinite index in $\mathcal R_1\times\mathcal R_2$. 
\end{lemma}
\begin{proof}
	Let $A\subseteq X_1\times X_2$ of positive measure, by Fubini's theorem for almost all $(x_1,x_2)\in A$ 
	the vertical section $A_{x_1}=\{x'_2\in X_2\colon (x_1,x'_2\in A)\}$ has positive measure, 
	and hence by aperiodicity the $\mathcal R_2$-class of $x_2$ intersects $A_{x_1}$ in an infinite set. 
	This implies that the $(\mathcal R_1\times \mathcal R_2)_{\restriction A}$-class of $(x_1,x_2)$ contains infinitely many $(\mathcal R_1\times \Delta_{X_2})_{\restriction A}$ classes.
\end{proof}

\begin{proposition}\label{prop: everywhere infinite index from product}
	The equivalence relation $\mathcal S$ from Example \ref{ex: diffuse example} is everywhere of infinite index in $\mathcal R_0$.
\end{proposition}
\begin{proof}
	The even-odd partition of $\N$ induces a bijection $\{0,1\}^\N\to\{0,1\}^\N\times \{0,1\}^\N$ under which $\mathcal R_0$ becomes $\mathcal R_0\times\mathcal R_0$ and $\mathcal S$ becomes $\mathcal R_0\times \Delta_{\{0,1\}^\N}$, so the result follows from the previous lemma.
\end{proof}
\subsection{Approximating the diagonal}\label{sec: converging to Delta}

Recall that $\Delta_X$ denotes the equality relation on $X$.
We will now characterize equivalence relations whose orbit closure contains $\Delta_X$, mirroring Popa's result on asymptotic orthogonalization of subalgebras of a finite factor \cite[Lem.\ 2.3]{popaAsymptoticOrthogonalizationSubalgebras2019}. 

\begin{theorem}\label{thm: infinite index everywhere characterization}
	Let $\mathcal R$ be an aperiodic p.m.p.\ equivalence relation. Let $\mathcal S\in\Sub(\mathcal R)$ be a subequivalence relation. The following are equivalent:
	\begin{enumerate}[(i)]
		\item \label{item: infinite index} $\mathcal S$ has everywhere infinite index in $\mathcal R$;
		\item \label{item: fullgroup orbit} the closure of the $[\mathcal R]$-orbit of $\mathcal S$ contains $\Delta_X$;
		\item \label{item: Aut orbit} the closure of the $\Aut(\mathcal R)$-orbit of $\mathcal S$ contains $\Delta_X$.
	\end{enumerate}
\end{theorem}

\begin{proof}
	The implication \eqref{item: fullgroup orbit} $\implies$ \eqref{item: Aut orbit} is clear since $[\mathcal R]\leq \Aut(\mathcal R)$. \\
	
	Let us then check that \eqref{item: Aut orbit} implies \eqref{item: infinite index} by proving the contrapositive: 
	assuming that $\mathcal S$ does not have infinite index everywhere, we need to show that the $\Aut(\mathcal R)$-orbit of $\mathcal S$ does not contain $\Delta_X$.
	
	By assumption, we have a subset $A\subseteq X$ of positive measure such that the restriction $\mathcal S_{\restriction A}$ has finite index in $\mathcal R_{\restriction A}$. Shrinking $A$ further if necessary, we may assume the index of the restriction of $\mathcal S$ to $A$ in the restriction of $\mathcal R$ to $A$ is constant equal to $k$.
	Let $N\in\N$ such that $1/N<\mu(A)$. 
	
	By aperiodicity and Maharam's lemma with $f=\frac 1{2kN}$ (see Proposition \ref{prop: maharam}), we can partition $X$ in $2kN$ pieces of equal $\mathcal R$-conditional measure. Using Lemma \ref{lem: full group orbit on malg}, we thus have $B\subseteq X$ and 
	$\varphi_1,...,\varphi_{2kN}\in[[\mathcal R]]$ with domain $B$
	such that $\varphi_1=\id_B$ and the sets 
	$\varphi_1(B),...,\varphi_{2kN}(B)$ partition $X$. 
	
	Let $B_0$ be the set of $x\in B$ such that there are at least $k+1$ distinct indices $i\in\{1,\dots,2kN\}$ such that $\varphi_i(x)\in A$.
	Write $A_0=\{x\in A\colon \exists i, \varphi_i\inv(x)\in B_0\}$, then since $A_0$ is covered by the disjoint translates of $B_0$ we have $\mu(A_0)\leq 2kN\mu(B_0)$. 
	Letting $A_1=A\setminus A_0$, we again have that $A_1$ is covered by the disjoint translates of $B_1=B\setminus B_0$, but by definition for every $x\in B_1$ there are at most $k$ indices $i\in\{1,\dots,2kN\}$ such that $\varphi(x)\in A_1$, so that $$\mu(A_1)\leq k\mu(B_1)\leq \frac k{2kN}.$$
	
	Since $A=A_0\sqcup A_1$, we then have $\mu(A)\leq \frac k{2kN}+2kN\mu(B_0)$. But $1/N<\mu(A)$ so  $1/2N<2kN\mu(B_0)$ and hence
	$$\mu(B_0)\geq \frac 1{2kN^2}.$$
	Now for all $x\in B_0$, because the index of $\mathcal S_{\restriction A}$ in $\mathcal R_{\restriction A}$ is $k$, we must have two distinct $i,j$ such that $(\varphi_i(x),\varphi_j(x))\in\mathcal S$. In particular,
	$$\sum_{1\leq i< j\leq 2N}M(\varphi_i\inv\varphi_j\cap \mathcal S)\geq \frac 1{2kN^2}.$$
	The latter estimate will actually be valid for every translate of  $\mathcal S$ by $T\in\mathcal \Aut(\mathcal R)$ because $T(A)$ will still satisfy $1/N<\mu(T(A))$, so for every $T\in\Aut(\mathcal R)$, 
	$$\sum_{1\leq i< j\leq 2N}M(\varphi_i\inv\varphi_j\cap T\cdot \mathcal S)\geq \frac 1{2kN^2}.$$
	This inequality defines a closed set of subequivalence relations, but it is not satisfied by $\Delta_X$, which finishes the proof of \eqref{item: Aut orbit}$\Rightarrow$\eqref{item: infinite index}. \\
	
	We finally prove that \eqref{item: infinite index} implies \eqref{item: fullgroup orbit}.
	We fix $\mathcal S\in\Sub(\mathcal R)$ everywhere of infinite index.
	Let us also fix a uniquely generating sequence $(\varphi_k)$ of  moving partial involutions of $\mathcal R$ as provided by Proposition \ref{prop: unique moving invol}. Applying Proposition \ref{prop: exhaustion by finite and topology} to the sequence of $M$-finite measure subsets $\varphi_k$, it suffices to show that given some $k\in\N$ and $\epsilon>0$, we can find $T\in[\mathcal R]$ such that for all $i\in\{1,\dots,k\}$, 
	\[
	M((T\inv\cdot \mathcal S)\cap \varphi_i)<\epsilon.
	\] 
	We will actually do better and construct $T\in [\mathcal R]$ such that for all $i\in\{1,\dots,k\}$, $M((T\inv\cdot \mathcal S)\cap \varphi_i)=0$.
	Observe that it now suffices (and it is necessary) to build $T$ such that that for almost all $x\in \dom\varphi_i$ and all $i\in\{1,\dots,k\}$, 
	\[
	(T(x),T\varphi_i(x))\not\in\mathcal S.
	\]
	Our construction makes a crucial use of a result of Eisenmann and Glasner:
	letting $\Gamma=\{\gamma_n\colon n\in\N\}$ be a dense subgroup of $[\mathcal R]$, 
	by \cite[Prop.~1.19]{eisenmannGenericIRSFree2016} we have that $\Gamma$ acts highly transitively 
	on the orbit of almost every $x\in X$, and we may as well restrict ourselves to the set of all such $x$'s.
	By the definition of high transitivity, this means that
	for all $x\in X$, every partial bijection between finite subsets
	of $\Gamma x$ is the restriction of some $\gamma\in\Gamma$.

	We will now build $T$  as the increasing union of elements $\psi_n$ of the pseudo full group defined inductively as follows.
	
	First, $\psi_0$ is the restriction of $\gamma_0$ to the set of all $x\in X$ such that for all $i\in\{1,\dots,k\}$, we have $(\gamma_0x, \gamma_0\varphi_i(x))\not\in\mathcal S$. 
	Then, assuming $\psi_n$ has been built, we extend it as $\psi_{n+1}$ by letting $$\psi_{n+1}(x)=\gamma_{n+1}x$$ 
	if $x\notin \dom \psi_n$ and $\gamma_{n+1}x\notin \rng \psi_n$ and for all $i\in\{1,\dots,k\}$
	\begin{enumerate}[(a)]
		\item \label{cond:  psi phii defined} if $\varphi_i(x)\in\dom \psi_{n}$ then $(\gamma_{n+1} x,\psi_n\varphi_i(x))\not\in\mathcal S$;
		\item \label{cond:  psi phii undefined} if $\varphi_i(x)\notin\dom \psi_{n}$  then $(\gamma_{n+1} x, \gamma_{n+1}\varphi_i(x))\not\in\mathcal S$.
	\end{enumerate}
	Let $\psi=\bigcup_n \psi_n$, we now prove the following central claim.
	\begin{claim}
	The element of the pseudo full group $\psi$ has full domain. 
	\end{claim}
	\begin{cproof}
	Assume not, observe first that for almost all $x$ in the complement of the domain of $\psi$, the $\mathcal R$-class of $x$ intersects of the complement of the range of $\psi$. Indeed otherwise we have a positive measure $\mathcal R$-invariant set such that $\psi\inv$ takes it into a subset of itself of smaller measure, contradicting that $\mathcal R$ is measure-preserving. 
	
	Now take $x\in X\setminus \dom\psi$ as above. 
	Let $i_1,...,i_l\in\{1,\dots,k\}$ be the indices such that $\psi\varphi_{i_1}(x),\dots, \psi\varphi_{i_l}(x)$ is defined, denote by $j_1\dots j_{k-l}$ the remaining indices.
	Since $\mathcal S$ has infinite index in $X\setminus \rng \psi$, there are $z\in [x]_{\mathcal R} \cap X\setminus \rng \psi$ and pairwise distinct $z_1,\dots,z_{k-l}\in[x]_{\mathcal R}\cap X\setminus \rng \psi $ such that
	\begin{itemize} 
		\item for all $m\in\{1,\dots,l\}$  we have $(z,\psi\varphi_{i_m}(x))\not\in\mathcal S$;
		\item for all $m\in\{1,\dots,k-l\}$ we have $(z,z_m)\not\in\mathcal S$.
	\end{itemize}
	By high transitivity, there is $\gamma\in\Gamma$ such that 
	$$\gamma x=z \text{ and }\forall m\in \{1,\cdots,k-l\}, \quad \gamma\varphi_{j_m}(x)=z_m.$$ 
	This shows that almost all $x\in X\setminus \dom\psi$ there is $n\in\N$ such that 
	$\gamma_{n+1} x\not\in \rng \psi$ and for all $i\in\{1,\dots, k\}$, we have : 
	\begin{enumerate}[(a')]
		\item if $\varphi_i(x)\in \dom \psi $  then $(\gamma_{n+1}x, \psi\varphi_i(x))\not\in\mathcal S$;
		\item if $\varphi_i(x)\notin \dom \psi$  then $(\gamma_{n+1}x, \gamma_{n+1}\varphi_i(x))\not\in\mathcal S$.
	\end{enumerate}
	Let $n$ be the first integer such that the set of $x\in X\setminus \dom \varphi$ satisfying the above conditions is non null. By the definition of $\psi$, this set should be contained in the domain of $\psi_{n+1}$ and hence of $\psi$, a contradiction.\end{cproof}
	
	So $\psi$ is everywhere defined; since $\mathcal R$ is measure-preserving 
	this implies that it belongs to the full group, and we thus rather write it as $T= \psi$.
	We finally check that $T$ is as wanted.

	Let $x\in X$, define $n\geq -1$ as the least integer such that $x\in\dom\psi_{n+1}$, so that 
	$T(x)=\psi_{n+1}(x)=\gamma_{n+1}(x)$ by construction.
	Let $i\in\{1,\dots,k\}$.

	If $\varphi_i(x)\in\dom \psi_n$, then $(\gamma_{n+1}x,\psi_n\varphi_i(x))\not\in\mathcal S$ by Condition \eqref{cond: psi phii defined}, which means that
	$(T(x),T\varphi_i(x))\not\in\mathcal S$. 
	
	If $\varphi_i(x)\notin\dom\psi_n$, let $m$ be the least integer such that $\varphi_i(x)\in \dom \psi_{m+1}$.
	Then $m\geq n$ and $\psi_{m+1}\varphi_i(x)=\gamma_{m+1}\varphi_i(x)$ by construction. There are two possibilites: 
	\begin{itemize}
		\item Either $m=n$, then Condition \eqref{cond: psi phii undefined} guarantees $(\gamma_{n+1}x, \gamma_{n+1}\varphi_i(x))\not\in\mathcal S$ so that $(\gamma_{n+1}x,\psi_{n+1}\varphi_i(x))\not\in\mathcal S$ and hence $(T(x),T\varphi_i(x))\not\in\mathcal S$.
		\item Or  $m>n$, 
		but then $(\gamma_{m+1}\varphi_i(x), \psi_m\varphi_i\varphi_i(x))\not\in\mathcal S$ by Condition \eqref{cond: psi phii defined} applied to $x'=\varphi_i(x)$ and $n'=m$. The fact that $\varphi_i$ is involutive implies that $(\psi_{m+1}\varphi_i(x),\psi_m(x))\not\in\mathcal S$ and hence $(T\varphi_i(x),T(x))\notin\mathcal S$, which is equivalent to $(T(x),T\varphi_i(x))\not\in\mathcal S$.
	\end{itemize}
	Since in all cases we reached the desired conclusion $(T(x),T\varphi_i(x))\not\in\mathcal S$, the proof is finished.
\end{proof}

\subsection{Dense orbits in the space of hyperfinite subequivalence relations}

By definition a CBER is called \textbf{finite} when all its equivalence classes are finite,
and \textbf{hyperfinite} if it can be written as an increasing union of finite Borel subequivalence relations. 
In our measured context, we ignore null sets and thus use the following definition.

\begin{definition}
	A non-singular equivalence relation $\mathcal R$ on $(X,\mu)$ is called \textbf{hyperfinite} 
	when it admits a restriction to a full measure set which is hyperfinite in the above sense.
\end{definition}

Hyperfiniteness can then  be characterized in full group terms as follows: $\mathcal R$ is hyperfinite iff for all $T_1,\dots,T_n\in [\mathcal R]$ and $\epsilon>0$, after throwing away a set $X'$ of measure $\epsilon$, the equivalence relation generated by the restrictions $T_{1\restriction X\setminus X'},..., T_{n\restriction X\setminus X'}$ is finite. 

\begin{remark}
	The above characterization of hyperfiniteness is Dye's original notion of \emph{approximate finiteness} for full groups \cite{dyeGroupsMeasurePreserving1959} (see also \cite[Prop.\ 1.57]{lemaitreGroupesPleinsPreservant2014} or \cite[Lem.\ 10.4]{kechrisTopicsOrbitEquivalence2004} for a proof of this characterization).
\end{remark}

\begin{definition}
Given a non-singular equivalence relation $\mathcal R$ on $(X,\mu)$, let us denote by $\Subhyp(\mathcal R)$ its space of hyperfinite subequivalence relations.
\end{definition}
Using $\liminf$'s as in the proof of Theorem \ref{thm: sub R is Polish} and the fact that the class of  non-singular hyperfinite equivalence relations is stable under countable increasing unions, one can show that that $\Subhyp(\mathcal R)$ is closed in $\Sub(\mathcal R)$ (see \cite[Thm.\ 8.1]{kechrisSpaceMeasurepreservingEquivalence}). Let us observe that this can also be seen as a consequence of the characterization via approximate finiteness: if $\mathcal S$ is not hyperfinite as witnessed by some $\epsilon>0$ and $T_1,...,T_n$ its in full group, any $\mathcal S'$ which contains the $T_i$'s in its full group up to an $\epsilon/2n$ error will also fail to be hyperfinite.

Before we state our main result on dense orbits in $\Subhyp(\mathcal R)$, we need a preparatory well-known lemma on finite subequivalence relations.

\begin{lemma}\label{lem: same orbit statistics}
	Let $\mathcal R$ be an ergodic p.m.p.\ equivalence relation.
	Then two finite subequivalence relations $\mathcal S_1$ and $\mathcal S_2$ of $\mathcal R$ are in the same $[\mathcal R]$-orbit iff for all $n\in\N$
	\begin{equation}\label{eq: same orbit size statistics}
	\mu(\{x\in\N\colon \abs{[x]_{\mathcal S_1}}=n\})=\mu(\{x\in\N\colon \abs{[x]_{\mathcal S_2}}=n\}).
	\end{equation}
\end{lemma}
\begin{proof}
	The direct implication is clear. Assume conversely that \eqref{eq: same orbit size statistics}
	holds.
	Let us fix some notation by letting, for $i\in\{1,2\}$ and $n\geq 1$
	\[
	X^i_n=\{x\in\N\colon \abs{[x]_{\mathcal S_i}}=n\}.
	\]
	Our assumption then becomes: the equality $\mu(X^1_n)=\mu(X^2_n)$ holds for every $n\geq 1$.
	Let $<$ be a Borel linear order on $X$. We can then
	define for all $i\in\{1,2\}$ and $j\in \{1,\dots,n\}$ the Borel set 
	\[
	Y_{n,j}^i=\{x\in X_n^i \colon x \text{ is the }j\text{'th element of }[x]_{\mathcal S_i}\}
	\]
	For $i\in\{1,2\}$, and $j_1,j_2\in\{1,\dots,n\}$, we have $\varphi_{n,j_1,j_2}^i: Y_{n,j_1}^i\to Y_{n,j_2}^i$ defined by taking $x\in Y_{n,j_1}^i$ to the $j_2$'th element of its $\mathcal S_i$-class, then $\varphi_{n,j_1,j_2}^i\in[[\mathcal S_i]]\subseteq[[\mathcal R]]$. In particular since $\mathcal R$ is p.m.p., we have that $\mu(Y_{n,j_1}^i)=\mu(Y_{n,j_2}^i)$. 
	Since $X_n^i$ is partitioned by $(Y_{n,j}^i)_{j=1}^n$, we conclude
	 \[n\mu(Y_{n,1}^1)=\mu(X_n^1)=\mu(X_n^2)=n\mu(Y_{n,1}^2)\]
	so that $\mu(Y_{n,1}^1)=\mu(Y_{n,1}^2)$.
	
	Since $\mathcal R$ is ergodic, for every $n\in\N$ we can fix some $\varphi_n\in[[\mathcal R]]$ such that $\dom\varphi_n=Y_{n,1}^1$ and $\rng \varphi_n=Y_{n,1}^2$. We then extend simultaneously these $\varphi_n$ as $T\in[\mathcal R]$ by letting for every $n \geq 1$, $j\in\{1,\dots n\}$ and $x\in X_{n,j}^1$:
	$$T(x)=\varphi_{n,1,j}^2\varphi_n\varphi_{n,j,1}^1(x)$$
	In other words, given $x\in X$ whose $\mathcal S_1$-class has cardinality $n$, we look at the first element $y$ of the $\mathcal S_1$-class of $x$, and then $T(x)$ is the element of the $\mathcal S_2$-class of $\varphi_n(y)$ which is in the same position as $x$.
	Then by construction $T\cdot \mathcal S_1=\mathcal S_2$ as wanted.
\end{proof}

\begin{theorem}\label{thm: closure orbit contains hf}
	Let $\mathcal R$ be a p.m.p.\ ergodic equivalence relation, let $\mathcal S\in\Sub(\mathcal R)$ be aperiodic and have everywhere infinite index. Then the closure of the $[\mathcal R]$-orbit of $\mathcal S$ contains $\Subhyp(\mathcal R)$.
\end{theorem}
\begin{proof}
	It follows from the definition of hyperfiniteness that finite equivalence relations are dense in $\Subhyp(\mathcal R)$, so it suffices to approximate every finite subequivalence relation of $\mathcal R$ by an element of the $[\mathcal R]$-orbit of $S$. 
	Let $\mathcal R_0$ be such a finite equivalence relation. For each $n$ let $$X_n=\{x\in X: \abs{[x]_{\mathcal R_0}}=n\}.$$

	Because $\mathcal S$ is aperiodic, we can find an element of the $[\mathcal R]$-orbit of $\mathcal R_0$ contained in $\mathcal S$.
	Indeed, by Maharam's lemma we first have a partition of $X$ into pieces $(Y_{j,n})_{1\leq j\leq n}$ such that 
	$\mu_{\mathcal S}(Y_{j,n})=\frac{\mu(X_n)}n$.
	We then have for every $2\leq j\leq n$ some $\varphi_{j,n}\in[[\mathcal S]]$ such that $\dom\varphi_{j,n}=Y_{1,n}$ and $\rng\varphi_{j,n}=Y_{j,n}$.

	Now let $Y_1=\bigsqcup_n Y_{1,n}$, $\psi_1=\id_{Y_1}$ and for all $j\geq 2$, $\psi_j=\bigsqcup_{n\geq j} \varphi_{j,n}$. Note that the ranges of the $\psi_{j}$'s partition $X$.
	Denoting by $\mathcal S_0$ the equivalence relation generated by all the $\psi_{j}$'s,  we then have 
	\[
	\mathcal S_0=\bigsqcup_n\bigsqcup_{i,j=1}^n \varphi_{n,i}\times\varphi_{n,j}(\Delta_{Y_{1,n}})= \bigsqcup_{i,j} \psi_{i}\times \psi_{j}(\Delta_{Y_{1}}\cap (\dom \psi_i\times\dom \psi_j)),
	\]
	In particular the $\mathcal S_0$-class of every $x\in \bigsqcup_{j=1}^n Y_{n,j}$ has cardinality $n$, and so by Lemma \ref{lem: same orbit statistics} we can fix some $T\in [\mathcal R]$ such that $T\cdot \mathcal S_0=\mathcal R_0$.

	By Theorem \ref{thm: infinite index everywhere characterization}, we find a sequence $(T_k)_k$ in the full group of the restriction of $\mathcal R$ to $Y_{1}$ such that $T_k\cdot \mathcal S_{\restriction Y_1}\to\Delta_{Y_1}$. 
	We then let $\tilde T_k(x)=\psi_{j}T_k\psi_{j}\inv(x)$ for all $x\in\rng\psi_{j}$. 
	
	By construction $\displaystyle \tilde T_k\cdot \mathcal S=\bigsqcup_{i,j} \psi_{i}\times\psi_{j}(T_k\cdot \mathcal S_{\restriction Y_{1}}\cap (\dom \psi_i\times\dom \psi_j))$, so by Lemma \ref{lem: continuous union on partition}
	$$\tilde T_k\cdot \mathcal S\to\bigsqcup_{i,j} \psi_{i}\times\psi_{j}(\Delta_{Y_{1}}\cap (\dom \psi_i\times\dom \psi_j))=\mathcal S_0$$
	which yields the desired result since we then have $T \tilde T_k\cdot\mathcal S\to T\cdot\mathcal S_0=\mathcal R_0$.
	\end{proof}

\begin{remark}
	By \cite[Cor.\ 5.4 (ii)]{ioanaSubequivalenceRelationsPositivedefinite2009}, if $\mathcal R$ is aperiodic and comes from a measure-preserving action of a property (T) countable group, then there are no dense orbits in the space of subequivalence relations because any $\mathcal S$ with a dense orbit would have to have finite index in some restriction of $\mathcal R$, and hence cannot contain $\Delta_X$ in its orbit by Theorem \ref{thm: infinite index everywhere characterization}.
	The general fact is that there cannot be dense orbits in the space of subequivalence relations of $\mathcal R$ as soon as $\mathcal R$ is not \emph{approximable} as defined by Gaboriau and Tucker-Drob in \cite{gaboriauApproximationsStandardEquivalence2016}.
	Indeed if $\mathcal S$ has a dense orbit Lemma \ref{lem: converging subsequence liminf} provides a sequence $\mathcal S_n$ such that $\mathcal R=\liminf \mathcal S_n$, while non approximability forces some restriction of $\bigcap_{n\geq N}\mathcal S_n$ to coincide with $\mathcal R$ on a positive measure subset $A$, contradicting the density of the orbit of $\mathcal S$ by Theorem \ref{thmi: infinite index everywhere characterization}.
	For more examples of non approximable equivalence relations, the reader can consult  the paper of Gaboriau and Tucker-Drob, where they obtain for instance a quantitative version of non-approximability for some equivalence relations coming from actions of product groups (see  \cite[Thm.~2.4]{gaboriauApproximationsStandardEquivalence2016}).
  \end{remark}

\begin{corollary}\label{cor: dense orbit in subhyp}
	Let $\mathcal R$ be an ergodic p.m.p.\ equivalence relation, let $\mathcal S\in\Sub(\mathcal R)$. The following are equivalent:
	\begin{enumerate}[(i)]
		\item\label{item: aper infind} $\mathcal S$ is aperiodic and has everywhere infinite index in $\mathcal R$;
		\item\label{item: aper closure fullgroup} $\mathcal S$ is aperiodic and the closure of the $[\mathcal R]$-orbit of $\mathcal S$ contains $\Delta_X$;
		\item\label{item: aper closure aut}$\mathcal S$ is aperiodic and the closure of the $\Aut(\mathcal R)$-orbit of $\mathcal S$ contains $\Delta_X$;
		\item\label{item: closure full group} The closure of the $[\mathcal R]$-orbit of $\mathcal S$ contains $\Subhyp(\mathcal R)$;
		\item\label{item: closure aut} The closure of the $\Aut(\mathcal R)$-orbit of $\mathcal S$ contains $\Subhyp(\mathcal R)$.
	\end{enumerate}
\end{corollary}
\begin{proof}
	The equivalence of the first three items \eqref{item: aper infind}, \eqref{item: aper closure fullgroup} and \eqref{item: aper closure aut} is a direct consequence of Theorem \ref{thm: infinite index everywhere characterization}. 
	The implication \eqref{item: aper infind}$\implies$\eqref{item: closure full group} is exactly Theorem \ref{thm: closure orbit contains hf}, and \eqref{item: closure full group} clearly implies \eqref{item: closure aut}.
	Finally, let us prove \eqref{item: closure aut}$\implies$\eqref{item: aper closure aut}: assume the closure of the $\Aut(\mathcal R)$-orbit of $\mathcal S$ contains $\Sub_{hyp}(\mathcal R)$.
	Since $\Delta_X$ is hyperfinite, we have that $\Delta_X$ belongs to the closure of the $\Aut(\mathcal R)$-orbit of $\mathcal S$, so we only have to show $\mathcal S$ is aperiodic.
	Assume by contradiction that $\mathcal S$ is not aperiodic, then for some $n\in\N$ we have $\mu(\{x\in X\colon \abs{[x]_{\mathcal S}}\leq n\})>0$.
	Let $\delta=\mu(\{x\in X\colon \abs{[x]_{\mathcal S}}\leq n\})$, then $\mathcal S$ belongs to the $\Aut(\mathcal R)$-invariant set $\mathbf B$ of all $\mathcal S'\in\Sub(\mathcal R)$ such that 
	$$\sup_{T_1,\dots,T_{n+1}\in[\mathcal R]}M(\mathcal S'\cap (T_1\cup\cdots\cup T_{n+1}))\leq \delta n+(1-\delta)(n+1).$$
	Note that the intersection map is continuous for the lower $M$-topology by Lemma \ref{lem: intersection is lower continous} and that $M$ is lower semi-continuous for the lower $M$-topology.
	Since arbitrary supremums of lower semi-continuous function are lower semi-continuous, 
	the $\Aut(\mathcal R)$-invariant set $\mathbf B$ is closed for the lower $M$-topology, in particular it is closed for the strong topology. Since the closure of the $\Aut(\mathcal R)$-orbit of $\mathcal S\in\mathbf B$ contains $\Subhyp(\mathcal R)$, we have $\mathbf B\subseteq \Subhyp(\mathcal R)$.
	
	 However $\mathbf B$ is disjoint from the set of aperiodic subequivalence relations since for $\mathcal T$ aperiodic we can find $T_1,...,T_{n+1}$ in the full group of $\mathcal T$ with disjoint graphs, so that $M(\mathcal T\cap (T_1\sqcup T_2\sqcup\cdots\sqcup T_{n+1}))=n+1$. Since $\Subhyp(\mathcal R)$ contains aperiodic subequivalence relations, we reached the desired contradiction. So $\mathcal S$ is aperiodic and hence \eqref{item: aper closure aut} holds as wanted.
\end{proof}

\begin{corollary}\label{cor: dense orbit in R0}
	For a subequivalence $\mathcal S$ of the hyperfinite ergodic p.m.p.\ equivalence relation $\mathcal R_0$, the following are equivalent:
	\begin{enumerate}[(i)]
		\item $\mathcal S$ is aperiodic and has everywhere infinite index in $\mathcal R_0$;
		\item $\mathcal S$ is aperiodic and the closure of the $[\mathcal R_0]$-orbit of $\mathcal S$ contains $\Delta_X$;
		\item $\mathcal S$ is aperiodic and the closure of the $\Aut(\mathcal R_0)$-orbit of $\mathcal S$ contains $\Delta_X$;
		\item The $[\mathcal R_0]$-orbit of $\mathcal S$ is dense in $\Sub(\mathcal R_0)$;
		\item The $\Aut(\mathcal R_0)$-orbit of $\mathcal S$ is dense in $\Sub(\mathcal R_0)$.
	\end{enumerate}
\end{corollary}
\begin{proof}
	This is a direct consequence of the previous corollary since every subequivalence relation of a hyperfinite equivalence relation is hyperfinite.
\end{proof}

\begin{remark}
	In the non-singular ergodic type II$_\infty$ or type III case, one can show that there are always dense orbits in the space of subequivalence relations.
	Let us sketch the proof: first note that by asumption there is a sequence $(\varphi_n)$ of elements of $[[\mathcal R]]$ such that for all $n$, $\dom \varphi_n=X$ but $X=\bigsqcup_n \rng \varphi_n$ (this actually characterizes equivalence relations of type II$_\infty$ or III among ergodic non-singular equivalence relations).
	We then enumerate a dense subset of $\Sub(\mathcal R)$ as $(\mathcal S_n)$.
	The desired subequivalence relation with a dense orbit is $\mathcal S=\bigsqcup_n \varphi_n\cdot \mathcal S_n$, as one can see by approximating each $\varphi_n$ by a full group element in the $d_M$ metric.
\end{remark}

\subsection{Meagerness of full group orbits}\label{sec: meager orbits}

In this section, we use Ioana's intertwining for subequivalence relations in order to  show that full groups orbits are always meager in the space of hyperfinite subequivalence relations. 
In what follows, we use the \textbf{uniform metric} on the full group of a p.m.p.\ equivalence relation $\mathcal R$, defined by 
$$d_u(T_1,T_2)=\mu(\{x\in X\colon T_1(x)\neq T_2(x)\})=d_M(T_1,T_2),$$
which is biinvariant, complete and separable, thus endowing $[\mathcal R]$ with a Polish group topology.
We will use without mention the well-known fact that if $M(\mathcal S\cap T)>1-\epsilon$, then there is $S\in[\mathcal S]$ such that $d_u(S,T)<\epsilon$ (to see this, note that $\mathcal S\cap T\in [[\mathcal S]]$, hence it can be extended to an element $S$ of the full group of $\mathcal S$ by Lemma \ref{lem: full group orbit on malg}).

\begin{definition}[Ioana]
	Let $\mathcal S, \mathcal T$ be two subequivalence relations of a p.m.p.\ equivalence relation $\mathcal R$, write $\mathcal S\prec\mathcal T$ if there is no sequence $(T_n)$ in the full group of $\mathcal T$ such that for all $U_1, U_2\in [\mathcal R]$, we have 
	\[
	M(\mathcal S\cap U_1T_nU_2)\to 0.
	\]	
\end{definition}

We will crucially use Ioana's version of Popa's interwining theorem 
\cite[Lem.~1.7]{ioanaUniquenessGroupMeasure2012}: if $\mathcal S\prec \mathcal T$
then $\mathcal S$ can be translated by an element of the full group of $\mathcal R$ so as to have somewhere
finite index in $\mathcal T$.
The interested reader is also refered to \cite[Lem.~3.1]{spaasStableDecompositionsRigidity2023} for a 
characterization of
$\mathcal S\prec\mathcal T$ when $\mathcal S$ has infinite index in $\mathcal R$.

We leave it to the reader to check that $\prec$ is $[\mathcal R]$-invariant, meaning that 
if $\mathcal S\prec\mathcal T$ and $U_1,U_2\in[\mathcal R]$, then $U_1\cdot \mathcal S\prec U_2\cdot \mathcal T$. Also note that we always have $\mathcal S\prec\mathcal S$.

\begin{lemma}\label{lem: intertw is Fsigma}
	The relation $\prec$ defines a $F_\sigma$ subset of $\Sub(\mathcal R)\times\Sub(\mathcal R)$.
\end{lemma}
\begin{proof}
	We show that the complement of $\prec$ in $\Sub(\mathcal R)\times\Sub(\mathcal R)$ is $G_\delta$.
	Let us fix $(U_i)$  dense in $\mathcal R$. 
	It is then not hard to check that $\mathcal S\not\prec\mathcal T$ iff there is a sequence $(T_n)$ in the full group of $T$ such that $M(\mathcal S\cap U_iT_nU_j)\to 0$ for all $i,j\in\N$ (this is essentially the first step of the proof of \cite[Lem.~1.7]{ioanaUniquenessGroupMeasure2012}). 
	It follows that $\mathcal S\not\prec\mathcal T$ if and only if for every $k\in\N$ and $\epsilon>0$, we can find $T\in[\mathcal T]$ such that 
	\begin{equation}\label{eq: intertwining epsilon}
	\forall i,j\in\{1,\dots,k\},\quad M(\mathcal S\cap U_i T U_j)<\epsilon.
	\end{equation}
	Now by density of $(U_i)$, we finally have $\mathcal S\not\prec\mathcal T$ iff for every $k\in\N$ and $\epsilon>0$, there is $l\in\N$ such that 
	$M(\mathcal T\cap U_l)>1-\epsilon$ and for all $i,j\in\{1,\dots,k\}$:
	\[
	M(\mathcal S\cap U_iU_lU_j)<\epsilon.
	\]
	It is now straightforward to check that this last condition defines a $G_\delta$ set, so we are done.
\end{proof}

\begin{theorem}\label{thm: meager orbits}
	Let $\mathcal R$ be an ergodic p.m.p.\ equivalence relation, consider the $[\mathcal R]$-action on the Polish space $\Subhyp(\mathcal R)$ of hyperfinite subequivalence relations of $\mathcal R$.
	Then all $[\mathcal R]$-orbits in $\Subhyp(\mathcal R)$ are meager.
\end{theorem}
\begin{proof}
	Assume by contradiction there is a non meager $[\mathcal R]$-orbit. Since there is a dense $[\mathcal R]$-orbit in $\Sub(\mathcal R)$, 
	the topological $0$-$1$ law yields that there is a comeager orbit.
	Denoting by $E$ the equivalence relation generated by the $[\mathcal R]$-action on $\Sub(\mathcal R)$, we deduce that $E$ is comeager in $\Subhyp(\mathcal R)\times\Subhyp(\mathcal R)$.
	
	By \cite[Thm.\ 4]{dyeGroupsMeasurePreserving1959}, the p.m.p.\ equivalence relation $\mathcal R$ contains an ergodic hyperfinite subequivalence relation $\mathcal S$.
	Let $\mathcal T$ be an aperiodic subequivalence relation of $\mathcal S$ with diffuse ergodic decomposition such as the one coming from example \ref{ex: diffuse example} once we identify $\mathcal S$ to $\mathcal R_0$. The following claim will essentially finish our proof.
	
	\begin{claim}
		We have  $\mathcal S\not\prec\mathcal T$.
	\end{claim}
	\begin{cproof}
	Suppose $\mathcal S\prec\mathcal T$. Then by \cite[Lem.\ 1.7]{ioanaUniquenessGroupMeasure2012}, there is a nonzero $\varphi\in[[\mathcal R]]$ and $k\in\N$ such that 
 	if $A=\dom\varphi$ and $B=\rng\varphi$, then every $\varphi\cdot \mathcal S_{\restriction A}$-class is contained in the union of at most $k$ classes of $\mathcal T_{\restriction B}$.
 	Since $\mathcal T$ has diffuse ergodic decomposition, we can find $C_1,\dots,C_{k+1}$ partitioning $B$ which are all $\mathcal T_{\restriction B}$-invariant and have measure $\frac 1{(k+1)\mu(B)}$. 
 	
 	Since $\mathcal S$ is ergodic, the $\varphi\cdot\mathcal S_{\restriction A}$-class of almost every $
	x\in B$ intersects all the $C_i$, and since they are $\mathcal T_{\restriction B}$-invariant, we conclude that almost every $\varphi\cdot \mathcal S_{\restriction A}$-class cannot be  contained in less than $k+1$ classes of $\mathcal T_{\restriction B}$, a contradiction.
	\end{cproof}
	
	Since $\prec$ is $[\mathcal R]$-invariant, its complement also is. This complement is moreover $G_\delta$ by Lemma \ref{lem: intertw is Fsigma}, and by the above claim it contains the subset $[\mathcal R]\cdot \mathcal S\times [\mathcal R]\cdot \mathcal T$, which is dense in $\Subhyp(\mathcal R)\times\Subhyp(\mathcal R)$ by Theorem \ref{thm: closure orbit contains hf}. 
	So the complement of $\prec$ is comeager, and it should thus intersect $E$. 
	This means that one can find a subequivalence relation $\mathcal S$ and $T\in[\mathcal R]$ such that $\mathcal S\not\prec T\cdot \mathcal S$, a contradiction.
\end{proof}
\begin{remark}
	Our proof is inspired by the Glasner-Weiss proof that $\Aut(X,\mu)$ does not have comeager conjugacy classes, replacing Rokhlin's lemma by Theorem \ref{thm: closure orbit contains hf} and Del Junco's disjointness by Ioana's interwining relation $\not\prec$ (see \cite[Sec.\ 3]{glasnerTopologicalGroupsRokhlin2008}). 
\end{remark}
\begin{corollary}\label{cor: orbits meager R0}
	Let $\mathcal R_0$ be the hyperfinite ergodic p.m.p.\ equivalence relation. Then all the 
	$[\mathcal R_0]$-orbits are meager in $\Sub(\mathcal R_0)$
\end{corollary}
\begin{proof}
	Again this follows directly from the above theorem since every subequivalence relation of a hyperfinite 
	equivalence relation is hyperfinie.
\end{proof}
\section{The uniform metric and complexity calculations}\label{sec: uniform}

We begin this final section by introducing a natural metric on the space of all non-singular equivalence relations. 
Given two non-singular equivalence relations $\mathcal R_1$ and $\mathcal R_2$, denote by $\mathfrak C(\mathcal R_1,\mathcal R_2)$ the set of all $A\in\MAlg(X,\mu)$ such that $\mathcal R_{1\restriction A}=\mathcal R_{2\restriction A}$. Define
$$d_u(\mathcal R_1,\mathcal R_2)=1-\sup_{A\in\mathfrak C(\mathcal R_1,\mathcal R_2)}\mu(A)=\inf_{A\in \mathfrak C(\mathcal R_1,\mathcal R_2)}1-\mu(A).$$
Let us check that this is indeed a metric. Symmetry is clear.
For Hausdorffness, if $d_u(\mathcal R_1,\mathcal R_2)=0$ then we find $(X_n)$ with $\mu(X_n)\geq 1-2^{-n}$ and $\mathcal R_{1\restriction X_n}=\mathcal R_{2\restriction X_n}$. By the Borel-Cantelli lemma almost every $x\in X$ is in all but finitely many $X_n$, which yields a full measure set restricted to which $\mathcal R_1$ and $\mathcal R_2$ coincide.

We finally need to show that the triangle inequality holds. Observe that
if $A_1\in \mathfrak C(\mathcal R_1,\mathcal R_2)$ and $A_2\in \mathfrak C(\mathcal R_2,\mathcal R_3)$, then $A_1\cap A_2\in \mathfrak C(\mathcal R_1,\mathcal R_3)$, moreover
$$X\setminus (A_1\cap A_2)\subseteq (X\setminus A_1)\cup (X\setminus A_2),$$ so by taking measures and infimums we get the desired triangle inequality.
\begin{lemma}
	The uniform metric refines the strong topology. 
\end{lemma}
\begin{proof}
The uniform metric clearly refines the uniform topology defined in \cite[Sec.\ 4.6]{kechrisSpaceMeasurepreservingEquivalence}, which in turn refines the strong topology.
\end{proof}

\begin{remark}
	The results from \cite[Thm.\ 1]{ioanaSubequivalenceRelationsPositivedefinite2009} applied to the subequivalence relation $\mathcal R_1\cap \mathcal R_2$ yield that the uniform metric actually induces the uniform topology. 
\end{remark}

In what follows, we freely identify the Cantor space $\{0,1\}^\N$ to the set of subsets of $\N$, 
and the continuous reduction alluded to is thus defined on the whole $\{0,1\}^\N$.

\begin{proposition} \label{prop: continuous reduction}
	Suppose $\mathcal R$ is an aperiodic non-singular equivalence relation.
	There is a continuous reduction from the set $\Subfin(\N)$ of finite subsets of $\N$ 
	to the complement of the set of subequivalence relations which are of infinite index, 
	and to the set of finite index subequivalence relations, both for the uniform metric.
\end{proposition}
\begin{proof}
	Let $(X_n)$ be a partition of $X$ into sets which intersect almost every $\mathcal R$-class 
	(e.g. obtained via Maharam's lemma as sets of $\mathcal R$-conditional measure  constant equal to $2^{-n-1}$). 
	Let $B\subseteq\N$. We associate to $B$ the set $Y_B=\bigcup_{n\in \N\setminus B} X_n$ 
	and let the reduction be $B\mapsto \mathcal R_B$ where
	$$\mathcal R_B=\left(\mathcal R\cap Y_B\times Y_B\right)\sqcup\bigcup_{n\in B} \mathcal R\cap (X_n\times X_n)$$ 
	The continuity of our reduction for the uniform metric is clear from the definition 
	(if two subsets $B$ and $B'$ coincide on the first $n$ integers, 
	then the associated equivalence relations coincide on  $\bigcup_{k<n}X_k$). 
	Moreover, if $B$ is finite we see that $\mathcal R_B$ has finite index equal to $\abs B+1$, 
	and if not it has infinite index as wanted. 
\end{proof}

\begin{corollary}
	Let $\mathcal R$ be a non-singular aperiodic equivalence relation.
	The set of subequivalence relations of $\mathcal R$ with infinite index is $G_\delta$-hard in the uniform metric, in particular it is $G_\delta$-complete in the strong topology.
\end{corollary}

The next statements are all about the strong topology.

\begin{proposition}
	Let $\mathcal R$ be a nonsingular aperiodic equivalence relation.
	The space of finite index subequivalence relations is $F_{\sigma\delta}$-hard if and only if $E$ has infinitely many ergodic components, otherwise it is $F_\sigma$-complete.
\end{proposition}
\begin{proof}
	Denote by $\Sub_{[<\infty]}(\mathcal R)$ the space of finite index subequivalence relations.
	Note that by ergodicity, if $E$ has only finitely many ergodic components, then the space of finite index subequivalence relations is equal to the union over $k\in\N$ of the spaces $\Sub_{[\leq k]}(\mathcal R)$ of all subequivalence relations whose index is uniformly less than $k$. 
	
	Let us show that for each $k\in\N$, the set $\Sub_{[\leq k]}(\mathcal R)$ is closed.
	Suppose $\mathcal S\notin \Sub_{[\leq k]}(\mathcal R)$, then there is a positive measure set of  $x\in X$ such that $[x]_{\mathcal R}$ contains at least $k+1$ distinct $\mathcal S$ classes. 
	We thus find $A\subseteq X$ of positive measure and $\varphi_1,...\varphi_{k}\in[[\mathcal R]]$ with common domain $A$ and such that $A,\rng \varphi_1,\dots,\rng\varphi_k$ are and $\mathcal S$ is disjoint from the union of the graphs of the $\varphi_i\varphi_j\inv$. In other words we find a finite subequivalence relation of $\mathcal R$ which is disjoint from $\mathcal S$ (after of course removing $\Delta_X$) and all whose non trivial classes have cardinal $\geq k+1$. 
	
	Now let $\mathcal S_n\to \mathcal S$,
	since 
	\[
	M(\varphi_1\sqcup\cdots \sqcup \varphi_k \cap \mathcal S_n)=\int_A \abs{\{i\in\{1,\dots,k\}\colon (x,\varphi_i(x))\in\mathcal S_n\}}d\mu(x)\to 0,
	\]
	 for $n$ large enough there will be a positive measure set of $x\in A$ such that for all $i\in\{1,\dots,k\}$, $(x,\varphi_i(x))\notin\mathcal S_n$, in particular $\mathcal S_n$ has somewhere index $>k$.
	 So $\Sub_{[\leq k]}(\mathcal R)$ is closed and hence $\Sub_{[<\infty]}(\mathcal R)=\bigcup_{k\in\N}\Sub_{[\leq k]}(\mathcal R)$ is $F_\sigma$. Moreover, $\Sub_{[<\infty]}(\mathcal R)=\bigcup_{k\in\N}\Sub_{[\leq k]}(\mathcal R)$ is $F_\sigma$-complete by virtue of the preceding proposition.
	
	Now if $\mathcal R$ has infinitely many ergodic components, let $(X_n)$ be a partition of $X$ into $\mathcal R$-invariant sets of positive measure.
	Denote by $\Subfin(\N)\subseteq \{0,1\}^\N$ the set of finite subsets of $\N$, viewed as finitely supported functions $\N\to\{0,1\}$.
	
	For every $n\in\N$, let $\Phi_n: \{0,1\}^\N\to \Sub(\mathcal R_{\restriction X_n})$ be the continuous reduction from 
	Proposition \ref{prop: continuous reduction}, and for a sequence $(x_n)_{n\in\N}$ of elements of $\{0,1\}^\N$,
	let $$\Phi((x_n)_{n\in\N})=\bigsqcup_{n\in\N} \Phi_n(x_n)\in\Sub(\mathcal R).$$ 
	By Lemma \ref{lem: continuous union on partition} $\Phi$ is continuous, and by construction 
	$$\Phi\inv (\Sub_{[<\infty]}(\mathcal R))=\prod_{n\in\N} \Subfin(\N).$$
	Being a (countable) infinite product of $F_\sigma$-hard sets, the latter is $F_{\sigma\delta}$-hard, so the set of finite index subequivalence relations is $F_{\sigma\delta}$-hard.
\end{proof}

\begin{remark}
	It follows from \cite[Prop.\ 9.4]{kechrisSpaceMeasurepreservingEquivalence} that $\Sub_{[<\infty]}(\mathcal R)$ is always $F_{\sigma\delta}$ (the proof in the nonsingular case works the same), so when $\mathcal R$ is aperiodic and has infinitely many ergodic components, $\Sub_{[<\infty]}(\mathcal R)$ is actually $F_{\sigma\delta}$-complete.
\end{remark}

We finally use a similar approach to show that  the set of finite subequivalence relations is always $F_{\sigma\delta}$-complete (for $\mathcal R$ aperiodic), by first directly showing that the space of finite subequivalence relations if $F_\sigma$-hard. In the remainder of the paper, we set $$\Gamma_0=\bigoplus_\N\Z/2\Z=\{f:\N\to\Z/2\Z\,\vert\, f(n)=0\text{ for all but finitely many }n\in\N\}.$$

\begin{lemma}
	 The space $\Sub_{\mathrm{fin}}(\Gamma_0)$ of finite subgroups of $\Gamma_0$ is an $F_\sigma$-complete subset of the space $\Sub(\Gamma_0)$ of subgroups of $\Gamma_0$. 
\end{lemma}
\begin{proof}
	The set of finite subgroups is countable, in particular it is $F_\sigma$. It is $F_\sigma$-complete via the continuous reduction $B\subseteq\N\mapsto \Lambda_B:=\bigoplus_{n\in B}\Z/2\Z\leq\Gamma$ of the $F_\sigma$-complete set of finite subsets of $\N$ to $\Sub_{\mathrm{fin}}(\Gamma_0)$. 
\end{proof}

\begin{remark}
	The above lemma can also be  proven by first noting that by the Baire category theorem, the countable dense set $\Sub_{\mathrm{fin}}(\Gamma_0)$ cannot be $G_\delta$ in the perfect zero-dimensional Polish space of subgroups of $\Gamma_0$, and then applying Wadge's theorem (see \cite[Thm.\ 22.10]{kechrisClassicalDescriptiveSet1995}).
\end{remark}

The following proposition was stated and proven in the p.m.p.\ setup in the first version of this paper
(in this restricted setup, it is a consequence of \cite[Thm.\ 4]{dyeGroupsMeasurePreserving1959}). 
We are very grateful to the referee for pointing out that 
the same result is true in the purely Borel setup and for indicating the proof.

\begin{proposition}\label{prop: free Gammanot action}
	For every aperiodic CBER $\mathcal R$, there is a free $\Gamma_0$-action
	which induces a subequivalence relation of $\mathcal R$. 
\end{proposition}
We rely on the following key lemma.
\begin{lemma}
	Let $\mathcal R$ be an aperiodic CBER, then its Borel full group contains 
	a fixed-point free involution. 
\end{lemma}
\begin{proof}
	By \cite[Lem.~7.4]{kechrisTopicsOrbitEquivalence2004}, 
	$\mathcal R$ contains a Borel subequivalence relation $\mathcal S$, all
	whose classes have cardinality $2$. 
	Our involution $U$ is then defined by: for all $x\in X$, $U(x)$ is the 
	only $y\neq x$ such that $(x,y)\in\mathcal S$.
\end{proof}

\begin{proof}[Proof of Proposition \ref{prop: free Gammanot action}]
	We first inductively define a decreasing sequence $(A_n)_{n\in\N}$
	of Borel subsets of $X$ and fixed point-free partial involutions $(\varphi_n)_{n\in\N}$
	in the Borel pseudo full group of $\mathcal R$ with support $A_n$ as follows.
	We begin with an everywhere defined involution $\varphi_0$
	with support $A_0=X$ as provided by the previous lemma and then, assuming $\varphi_n$ has been built, 
	we let $A_{n+1}$ be a Borel fundamental domain for $\varphi_n$. Noting 
	that the restriction of $\mathcal R$ to $A_{n+1}$ has to remain aperiodic,
	we use the above lemma to find $\varphi_{n+1}$ as a fixed point-free involution
	in the Borel full group of the restriction of $\mathcal R$ to $A_{n+1}$.

	We will extend these to involutions so as to have a free $\Gamma_0$-action. Letting $G_n=\{ f\in \bigoplus_{\N}\Z/2\Z\colon \forall m> n, f(m)=0\}$, we can write $\Gamma_0$ as the increasing union $\Gamma_0=\bigcup_n G_n$, and we have 
	$G_n=\la s_0,\dots,s_{n}\ra$, where $s_i(k)=0$ if $i\neq k$, and $s_i(k)=1$ if $i=k$.
	Our inductive construction of the action will be so that for all $n$, 
	\begin{equation}\label{eq: prescribed fd}
		X=\bigsqcup_{g\in G_n} g\cdot A_{n+1}
	\end{equation}
	
	We begin by letting $s_0\cdot x=\varphi_0(x)$ for all $x\in X$.
	And then,  assuming that the action of $G_n$ has been defined so that $X=\bigsqcup_{g\in G_n} g\cdot A_{n+1}$, 
	we start defining the $s_{n+1}$ action on $A_{n+1}$ by letting 
	$s_{n+1}\cdot x=\varphi_{n+1}(x)$ if $x\in A_{n+1}$. 
	Since we want $s_{n+1}$ to commute with $G_n$ and we have $X=\bigsqcup_{g\in G_n} g\cdot A_{n+1}$, 
	we are then forced to extend $s_{n+1}$ to the whole of $x$ by letting, for every $x\in A_{n+1}$ 
	and every $g\in G_n$, 
	\[ 
	s_{n+1}(g\cdot x)=g\cdot s_{n+1}\cdot x. 
	\]
	It is then not hard to check that this does define a $G_{n+1}$ action with 
	$X=\bigsqcup_{g\in G_{n+1}} g\cdot A_{n+2}$, which provides us a $\Gamma_0$-action by induction in the end. 
	This action is free since Equation \eqref{eq: prescribed fd} implies that the action restricted to each $G_n$ is free, 
	which finishes the proof.
\end{proof}

\begin{remark}
	It follows from \cite[Thm.\ 3]{dyeGroupsMeasurePreserving1959} that when $\mathcal R$ is p.m.p.\ hyperfinite 
	and aperiodic, one can upgrade the above result and obtain that $\mathcal R$ 
	is equal to the equivalence relation generated by the 
	$\Gamma_0$-action. 
	This led Dye to formulate the question whether any 
	p.m.p.\ equivalence relation should come the free action of some countable group 
	(see the paragraph right before Lemma 6.5 in \cite{dyeGroupsMeasurePreserving1959}). 
	This question was answered in the negative 40 years later by Furman \cite{furmanOrbitEquivalenceRigidity1999}.
\end{remark}
Given a non-singular equivalence relation $\mathcal R$, we denote by $\Subfin(\mathcal R)$ the space of finite subequivalence relations of $\mathcal R$.

\begin{proposition}
	Let $\mathcal R$ be an aperiodic non-singular equivalence relation. 
	Then $\Subfin(\mathcal R)$ is $F_{\sigma\delta}$-complete. 
\end{proposition}
\begin{proof}
	Partition $X$ into $(X_n)_{n\in\N}$ so that each $X_n$ has positive measure and the restriction of $\mathcal R$ to each 
	$X_n$ is aperiodic. 
	By the previous lemma each restriction of $\mathcal R$ to $X_n$ contains a subequivalence relation 
	coming from a free $\Gamma_0$-action $\alpha_n$. 
	Now consider the map $(\mathrm{Sub}(\Gamma_0))^\N\to \Sub(\mathcal R)$ which takes any sequence $(\lambda_n)_{n\in\N}$ 
	to the equivalence relation $\bigsqcup_{n\in\N}\mathcal R_{\alpha_n(\Lambda_n)}$. 
	This is a topological embedding by Lemma \ref{lem: continuous union on partition}, 
	and it reduces the $F_{\sigma\delta}$-complete set $\Sub_{\mathrm{fin}}(\Gamma_0) ^\N\subseteq \Sub(\Gamma_0)^\N$ to the set of finite equivalence relations, which is thus $F_{\sigma\delta}$-hard.

	We are left with proving that $\Subfin(\mathcal R)$ is $F_{\sigma\delta}$, which is due to 
	Kechris in the p.m.p.\ setup \cite[Thm.\ 8.5]{kechrisSpaceMeasurepreservingEquivalence}. 
	While his proof generalizes to the non-singular setup,
	we provide here another approach relying on spaces of measurable maps.
	Let us fix a uniquely generating sequence of moving partial involutions $(\varphi_n)_{n\in\N}$, 
	consider for every $n\in\N$ the continuous map 
	\[
	\iota_n: \Sub(\mathcal R)\to \MAlg(X,\mu)
	\]
	which associates to every $\mathcal S\in\Sub(\mathcal R)$ 
	the set of $x\in \dom\varphi_n$ such that $(x,\varphi_n(x))\in \mathcal S$.

	Identifying subsets to characteristic functions yields a natural homeomorphism $\Phi$
	between $\MAlg(X,\mu)$ and the Polish space of measurable maps 
	$\LL^0(X,\mu,\{0,1\})$ (see \cite[Appendix B]{lemaitreL1FullGroups2025}
	for more on $\LL^0$ spaces).
	Using this homeomorphism and putting all the maps $\iota_n$ together, we get an continuous embedding
	\[
	\iota=(\Phi\circ \iota_n)_{n\in\N}: \Sub(\mathcal R)\into \prod_{n\in\N}\LL^0(X,\mu,\{0,1\})\simeq \LL^0(X,\mu,\{0,1\}^\N). 
	\]
	We will conclude the proof via a straightforward application of the claim below.
	\begin{claim}
		Let $Y$ be a Polish space, let $D\subseteq Y$ be $F_\sigma$. 
		Then $\LL^0(X,\mu,D)$ is an $F_{\sigma\delta}$ subset of $\LL^0(X,\mu,Y)$. 
	\end{claim}
	\begin{cproof}
	It follows directly from the definition of the topology of convergence in measure given in 
	\cite[Appendix~B]{lemaitreL1FullGroups2025} that for any open subset 
	$U\subseteq Y$,
	the map $f\in\LL^0(X,\mu,Y)\mapsto \mu(f\inv(U))$ is lower semi-continuous.
	Taking complements, given a closed subset $C\subseteq\{0,1\}^\N$, the map 
	$f\in\LL^0(X,\mu,Y)\mapsto \mu(f\inv(C))$ is thus upper semi-continuous. 
	Now write $D=\bigcup_n C_n$ as an increasing union 
	of closed subsets $C_n$, and observe that $\mu(f\inv (D)))=1$ iff for all $\epsilon>0$ 
	there is $n$ such that $\mu(f\inv(C_n))\geq 1-\epsilon$. 
	\end{cproof}
	
	In order to finish the proof, take $D=\Subfin(\N)$ to be the $F_\sigma$ set of finite subsets of $\N$,
	viewed as a subset of $Y=\{0,1\}^\N$.
	Then $\LL^0(X,\mu,\Subfin(\N))$ is $F_{\sigma\delta}$ in $\LL^0(X,\mu,\{0,1\}^\N)$ by the above claim. 
	Finally, observe that
	$\Subfin(\mathcal R)=\iota\inv(\LL^0(X,\mu,\Subfin(\N)))$ which is thus $F_{\sigma\delta}$ 
	by continuity of $\iota$.
\end{proof}

\begin{remark}
	Endowing $\Sub(\mathcal R)$  with its strong topology and the uniform metric, 
	we get a Polish topometric space in the sense of Ben Yaacov \cite{benyaacovTopometricSpacesPerturbations2008}, 
	and our remark preceding Corollary  \ref{cor: dense orbit in subhyp} can be upgraded to the fact that if 
	a p.m.p.~equivalence relation
	$\mathcal R$ has property (T), then $\mathcal R$ is a metrically isolated point of $\Sub(\mathcal R)$ (see also \cite{gaboriauApproximationsStandardEquivalence2016} for more examples). I don't know whether the action of the full group or automorphism group of the hyperfinite ergodic p.m.p.\ equivalence relation admits a metrically generic orbit, as characterized in \cite{benyaacovTopometricEffrosTheorem2023}.
\end{remark}

\bibliographystyle{alphaurl}
\bibliography{zoterobib}

\begin{thebibliography}{FMMP24}

\bibitem[Aoi03]{aoiConstructionEquivalenceSubrelations2003}
Hisashi Aoi.
\newblock A construction of equivalence subrelations for intermediate subalgebras.
\newblock {\em Journal of the Mathematical Society of Japan}, 55(3), 2003.
\newblock \href {https://doi.org/10.2969/jmsj/1191418999} {\path{doi:10.2969/jmsj/1191418999}}.

\bibitem[Bou98]{bourbakiGeneralTopologyChapters1998}
N.~Bourbaki.
\newblock {\em General {{Topology}}: {{Chapters}} 1--4}.
\newblock {Springer-Verlag Berlin and Heidelberg GmbH \& Co. K}, Berlin Heidelberg, 2nd edition, 1998.

\bibitem[BY08]{benyaacovTopometricSpacesPerturbations2008}
Ita{\"i} Ben~Yaacov.
\newblock Topometric spaces and perturbations of metric structures.
\newblock {\em Logic and Analysis}, 1(3):235, 2008.
\newblock \href {https://doi.org/10.1007/s11813-008-0009-x} {\path{doi:10.1007/s11813-008-0009-x}}.

\bibitem[BYM23]{benyaacovTopometricEffrosTheorem2023}
Ita{\"i} Ben~Yaacov and Julien Melleray.
\newblock A topometric {{Effros}} theorem.
\newblock {\em The Journal of Symbolic Logic}, pages 1--11, 2023.
\newblock \href {https://doi.org/10.1017/jsl.2023.5} {\path{doi:10.1017/jsl.2023.5}}.

\bibitem[Coh13]{cohnMeasureTheorySecond2013}
Donald~L. Cohn.
\newblock {\em Measure {{Theory}}: {{Second Edition}}}.
\newblock Birkh{\"a}user {{Advanced Texts Basler Lehrb{\"u}cher}}. Springer, New York, NY, 2013.
\newblock \href {https://doi.org/10.1007/978-1-4614-6956-8} {\path{doi:10.1007/978-1-4614-6956-8}}.

\bibitem[Dye59]{dyeGroupsMeasurePreserving1959}
H.~A. Dye.
\newblock On {{Groups}} of {{Measure Preserving Transformations}}. {{I}}.
\newblock {\em American Journal of Mathematics}, 81(1):119--159, 1959.
\newblock \href {https://doi.org/10.2307/2372852} {\path{doi:10.2307/2372852}}.

\bibitem[EG16]{eisenmannGenericIRSFree2016}
Amichai Eisenmann and Yair Glasner.
\newblock Generic {{IRS}} in free groups, after {{Bowen}}.
\newblock {\em Proceedings of the American Mathematical Society}, 144(10):4231--4246, 2016.
\newblock \href {https://doi.org/10.1090/proc/13020} {\path{doi:10.1090/proc/13020}}.

\bibitem[FMMP24]{fimaMichaelsSelectionTheorem2024}
Pierre Fima, Fran{\c c}ois~Le Ma{\^i}tre, Kunal Mukherjee, and Issan Patri.
\newblock Michael's selection theorem and applications to the {{Mar{\'e}chal}} topology.
\newblock {\em arXiv preprint}, 2024.
\newblock \href {https://doi.org/10.48550/arXiv.2407.05776} {\path{doi:10.48550/arXiv.2407.05776}}.

\bibitem[Fur99]{furmanOrbitEquivalenceRigidity1999}
Alex Furman.
\newblock Orbit {{Equivalence Rigidity}}.
\newblock {\em Annals of Mathematics}, 150(3):1083--1108, 1999.
\newblock \href {https://doi.org/10.2307/121063} {\path{doi:10.2307/121063}}.

\bibitem[Gao09]{gaoInvariantDescriptiveSet2009}
Su~Gao.
\newblock {\em Invariant Descriptive Set Theory}, volume 293 of {\em Pure and {{Applied Mathematics}} ({{Boca Raton}})}.
\newblock CRC Press, Boca Raton, FL, 2009.

\bibitem[GL09]{gaboriauMeasurablegrouptheoreticSolutionNeumanns2009}
Damien Gaboriau and Russell Lyons.
\newblock A measurable-group-theoretic solution to von {{Neumann}}'s problem.
\newblock {\em Inventiones mathematicae}, 177(3):533--540, 2009.
\newblock \href {https://doi.org/10.1007/s00222-009-0187-5} {\path{doi:10.1007/s00222-009-0187-5}}.

\bibitem[GT16]{gaboriauApproximationsStandardEquivalence2016}
Damien Gaboriau and Robin {Tucker-Drob}.
\newblock Approximations of standard equivalence relations and {{Bernoulli}} percolation at pu.
\newblock {\em Comptes Rendus Mathematique}, 354(11):1114--1118, 2016.
\newblock \href {https://doi.org/10.1016/j.crma.2016.09.011} {\path{doi:10.1016/j.crma.2016.09.011}}.

\bibitem[GW08]{glasnerTopologicalGroupsRokhlin2008}
Eli Glasner and Benjamin Weiss.
\newblock Topological groups with {{Rokhlin}} properties.
\newblock {\em Colloquium Mathematicae}, 110(1):51--80, 2008.

\bibitem[IKT09]{ioanaSubequivalenceRelationsPositivedefinite2009}
Adrian Ioana, Alexander Kechris, and Todor Tsankov.
\newblock Subequivalence relations and positive-definite functions.
\newblock {\em Groups, Geometry, and Dynamics}, pages 579--625, 2009.
\newblock \href {https://doi.org/10.4171/GGD/71} {\path{doi:10.4171/GGD/71}}.

\bibitem[Ioa12]{ioanaUniquenessGroupMeasure2012}
Adrian Ioana.
\newblock Uniqueness of the {{Group Measure Space Decomposition}} for {{Popa}}'s ${{\mathcal{{HT}}}}$ {{Factors}}.
\newblock {\em Geometric and Functional Analysis}, 22(3):699--732, 2012.
\newblock \href {https://doi.org/10.1007/s00039-012-0178-3} {\path{doi:10.1007/s00039-012-0178-3}}.

\bibitem[Kec]{kechrisSpaceMeasurepreservingEquivalence}
Alexander~S. Kechris.
\newblock The space of measure-preserving equivalence relations and graphs.
\newblock {\em in preparation, preliminary version available on his webpage}.

\bibitem[Kec95]{kechrisClassicalDescriptiveSet1995}
Alexander~S. Kechris.
\newblock {\em Classical Descriptive Set Theory}, volume 156 of {\em Graduate {{Texts}} in {{Mathematics}}}.
\newblock Springer-Verlag, New York, 1995.
\newblock \href {https://doi.org/10.1007/978-1-4612-4190-4} {\path{doi:10.1007/978-1-4612-4190-4}}.

\bibitem[Kec10]{kechrisGlobalaspectsergodic2010}
Alexander~S. Kechris.
\newblock {\em Global Aspects of Ergodic Group Actions}, volume 160 of {\em Mathematical {{Surveys}} and {{Monographs}}}.
\newblock American Mathematical Society, Providence, RI, 2010.
\newblock \href {https://doi.org/10.1090/surv/160} {\path{doi:10.1090/surv/160}}.

\bibitem[KM04]{kechrisTopicsOrbitEquivalence2004}
Alexander Kechris and Benjamin~D. Miller.
\newblock {\em Topics in {{Orbit Equivalence}}}.
\newblock Lecture {{Notes}} in {{Mathematics}}. Springer-Verlag, Berlin Heidelberg, 2004.
\newblock \href {https://doi.org/10.1007/b99421} {\path{doi:10.1007/b99421}}.

\bibitem[LM14]{lemaitreGroupesPleinsPreservant2014}
Fran{\c c}ois Le~Ma{\^i}tre.
\newblock {\em {Sur les groupes pleins pr{\'e}servant une mesure de probabilit{\'e}}}.
\newblock PhD thesis, ENS Lyon, 2014.

\bibitem[LM16]{lemaitrefullgroupsnonergodic2016}
Fran{\c c}ois Le~Ma{\^i}tre.
\newblock On full groups of non-ergodic probability-measure-preserving equivalence relations.
\newblock {\em Ergodic Theory and Dynamical Systems}, 36(7):2218--2245, 2016.
\newblock \href {https://doi.org/10.1017/etds.2015.20} {\path{doi:10.1017/etds.2015.20}}.

\bibitem[LM22]{lemaitrePolishTopologiesGroups2022}
Fran{\c c}ois Le~Ma{\^i}tre.
\newblock Polish topologies on groups of non-singular transformations.
\newblock {\em Journal of Logic and Analysis}, 14, 2022.
\newblock \href {https://doi.org/10.4115/jla.2022.14.4} {\path{doi:10.4115/jla.2022.14.4}}.

\bibitem[LMS25]{lemaitreL1FullGroups2025}
Fran{\c c}ois Le~Ma{\^i}tre and Konstantin Slutsky.
\newblock $\mathrm{L}^1$ full groups of flows.
\newblock {\em To appear in Memoirs of the EMS}, 2025.
\newblock \href {https://doi.org/10.48550/arXiv.2108.09009} {\path{doi:10.48550/arXiv.2108.09009}}.

\bibitem[Mar73]{marechalTopologieStructureBorelienne1973}
Odile Mar{\'e}chal.
\newblock Topologie et structure bor{\'e}lienne sur l'ensemble des alg{\`e}bres de von {{Neumann}}.
\newblock {\em C. R. Acad. Sci. Paris Ser. I Math.}, 276:847--850, 1973.

\bibitem[OW80]{ornsteinErgodicTheoryAmenable1980}
Donald~S. Ornstein and Benjamin Weiss.
\newblock Ergodic theory of amenable group actions. {{I}}. {{The Rohlin}} lemma.
\newblock {\em American Mathematical Society. Bulletin. New Series}, 2(1):161--164, 1980.
\newblock \href {https://doi.org/10.1090/S0273-0979-1980-14702-3} {\path{doi:10.1090/S0273-0979-1980-14702-3}}.

\bibitem[Pop19]{popaAsymptoticOrthogonalizationSubalgebras2019}
Sorin Popa.
\newblock Asymptotic {{Orthogonalization}} of {{Subalgebras}} in {{II}}$_1$ {{Factors}}.
\newblock {\em Publications of the Research Institute for Mathematical Sciences}, 55(4):795--809, 2019.
\newblock \href {https://doi.org/10.4171/prims/55-4-5} {\path{doi:10.4171/prims/55-4-5}}.

\bibitem[Sew14]{sewardEveryActionNonamenable2014}
Brandon Seward.
\newblock Every action of a nonamenable group is the factor of a small action.
\newblock {\em Journal of Modern Dynamics}, 8(2):251--270, 2014.
\newblock \href {https://doi.org/10.3934/jmd.2014.8.251} {\path{doi:10.3934/jmd.2014.8.251}}.

\bibitem[Spa23]{spaasStableDecompositionsRigidity2023}
Pieter Spaas.
\newblock Stable decompositions and rigidity for products of countable equivalence relations.
\newblock {\em Transactions of the American Mathematical Society}, 376(03):1867--1894, 2023.
\newblock \href {https://doi.org/10.1090/tran/8800} {\path{doi:10.1090/tran/8800}}.

\bibitem[Zho24]{zhouNoncommutativeTopologicalBoundaries2024}
Shuoxing Zhou.
\newblock Noncommutative topological boundaries and amenable invariant random intermediate subalgebras.
\newblock {\em arXiv preprint}, 2024.
\newblock \href {https://doi.org/10.48550/arXiv.2407.10905} {\path{doi:10.48550/arXiv.2407.10905}}.

\end{thebibliography}

\end{document}